\documentclass[12pt]{amsart}
\usepackage{amssymb}
\usepackage{eucal}
\usepackage[total={6.25truein,9truein},centering]{geometry} 
\usepackage{hyperref}

\newtheorem{theorem}[equation]{Theorem}
\newtheorem{lemma}[equation]{Lemma}
\newtheorem{proposition}[equation]{Proposition}
\newtheorem{corollary}[equation]{Corollary}

\newtheorem*{theorem:hypersum}{Theorem~\ref{T:hypersum}}
\newtheorem*{theorem:Anderson}{Theorem~\ref{T:Anderson}}
\newtheorem*{theorem:Thakur}{Theorem~\ref{T:Thakur}}

\theoremstyle{definition}

\newtheorem{example}[equation]{Example}
\newtheorem*{acknowledgments}{Acknowledgments}

\theoremstyle{remark}
\newtheorem{remark}[equation]{Remark}

\numberwithin{equation}{section}

\allowdisplaybreaks

\newcommand{\FF}{\mathbb{F}}
\newcommand{\ZZ}{\mathbb{Z}}

\newcommand{\KK}{\mathbb{K}}

\newcommand{\bs}{\mathbf{s}}

\newcommand{\bw}{\mathbf{w}}
\newcommand{\bx}{\mathbf{x}}

\newcommand{\by}{\mathbf{y}}
\newcommand{\bz}{\mathbf{z}}

\newcommand{\cE}{\mathcal{E}}
\newcommand{\cH}{\mathcal{H}}
\newcommand{\cI}{\mathcal{I}}

\newcommand{\cN}{\mathcal{N}}

\newcommand{\cP}{\mathcal{P}}
\newcommand{\cS}{\mathcal{S}}

\newcommand{\cW}{\mathcal{W}}

\DeclareMathOperator{\Mat}{Mat}

\DeclareMathOperator{\Id}{Id}

\DeclareMathOperator{\sz}{sz}

\DeclareMathOperator{\wt}{wt}

\newcommand{\oi}{\overline{\imath}}

\newcommand{\oK}{\mkern2.5mu\overline{\mkern-2.5mu K}}

\newcommand{\ox}{\overline{x}}
\newcommand{\oy}{\overline{y}}
\newcommand{\ozero}{\overline{0}}

\newcommand{\uell}{\underline{\ell}}
\newcommand{\ueta}{\underline{\eta}}

\newcommand{\sep}{\mathrm{sep}}

\newcommand{\pd}{\partial}

\newcommand{\power}[2]{{#1 [[ #2 ]]}}
\newcommand{\laurent}[2]{{#1 (( #2 ))}}
\newcommand{\brac}[2]{\genfrac{\{}{\}}{0pt}{}{#1}{#2}}

\newcommand{\assign}{\mathrel{\vcenter{\baselineskip0.5ex \lineskiplimit0pt
                     \hbox{\scriptsize.}\hbox{\scriptsize.}}}%
                     =}
\newcommand{\rightassign}{=\mathrel{\vcenter{\baselineskip0.5ex \lineskiplimit0pt
                     \hbox{\scriptsize.}\hbox{\scriptsize.}}}}

\begin{document}

\title[Hyperderivative power sums and Carlitz multiplication coefficients]{Hyperderivative power sums, Vandermonde matrices, and Carlitz multiplication coefficients}

\author{Matthew A. Papanikolas}
\address{Department of Mathematics, Texas A{\&}M University, College Station,
TX 77843, U.S.A.}
\email{papanikolas@tamu.edu}


\thanks{This project was partially supported by NSF Grant DMS-1501362}


\date{October 27, 2020}

\begin{abstract}
We investigate interconnected aspects of hyperderivatives of polynomials over finite fields, $q$-th powers of polynomials, and specializations of Vandermonde matrices.  We construct formulas for Carlitz multiplication coefficients using hyperderivatives and symmetric polynomials, and we prove identities for hyperderivative power sums in terms of specializations of the inverse of the Vandermonde matrix.  As an application of these results we give a new proof of a theorem of Thakur on explicit formulas for Anderson's special polynomials for log-algebraicity on the Carlitz module.  Furthermore, by combining results of Pellarin and Perkins with these techniques, we obtain a new proof of Anderson's theorem in the general case.
\end{abstract}


\dedicatory{In honor and memory of David Goss}

\maketitle

\section{Introduction} \label{S:Intro}

Letting $A= \FF_q[\theta]$ be the polynomial ring in one variable over a finite field and $K = \FF_q(\theta)$ be its fraction field, it is natural to consider polynomial power sums in~$K$,
\begin{equation}
  S_i(k) = \sum_{a \in A_{i+}} a^k, \quad k \in \ZZ,
\end{equation}
where $A_{i+}$ denotes the finite set of monic polynomials in $A$ of degree~$i$.  Vanishing results for these series and special value formulas were obtained early on by Carlitz~\cite{Carlitz35}, \cite{Carlitz37}, and Lee~\cite{Lee43}, and they have been studied by many researchers in intervening years for both their intrinsic interest and their applications to values of Goss $L$-functions, multizeta values, and Anderson's log-algebraic identities (e.g., see~\cite{AnglesPellarin14}, \cite{Gekeler88a}, \cite{Goss78}, \cite{Thakur92}, \cite{Thakur}, \cite{Thakur09}).

As a variant on these types of problems the present paper considers hyperderivative power sums of the form
\begin{equation}
  H_i(j_1, \dots, j_\ell; k_1, \dots, k_{\ell}) = \sum_{a \in A_{i+}} \pd_{\theta}^{j_1}(a)^{k_1} \cdots \pd_{\theta}^{j_\ell}(a)^{k_{\ell}},
\end{equation}
for $j_1, \dots, j_{\ell} \geqslant 0$ and $k_1, \dots, k_{\ell} \in \ZZ$, where $\pd_{\theta}^j(a)^{k}$ is the $k$-th power of the $j$-th hyperderivative of $a$ with respect to~$\theta$ (see~\S\ref{S:Prelim} for the definitions).  Hyperderivatives have become increasingly important in research in function field arithmetic in recent years (e.g., see~\cite{BosserPellarin08}, \cite{BrownawellDenis00}, \cite{Conrad00}, \cite{Goss17}, \cite{Jeong00}, \cite{Jeong11}, \cite{Maurischat18}, \cite{PZeng17}).  These sums appear in log-algebraic identities for the Carlitz module (see~\S\ref{S:Thakur}--\ref{S:Anderson}), as well as for its tensor powers~\cite{PLogAlg}.

In order to investigate these sums and their applications, we first observe that there is substantial intertwining among $q$-th powers of polynomials, hyperderivatives, and symmetric polynomials.  For example, in \S\ref{S:Vandermonde} we show that for $a \in A$ of degree at most~$i$, if we write $a(t)$ for the polynomial obtained by replacing $\theta$ by another variable $t$, then
\begin{equation}
  \begin{pmatrix} a(t) \\ a(t)^q \\ \vdots \\ a(t)^{q^i} \end{pmatrix}
  = V_i \bigl( t-\theta, t^q - \theta, \dots, t^{q^i} -\theta \bigr)
  \begin{pmatrix} a \\ \pd_{\theta}^1(a) \\ \vdots \\ \pd_{\theta}^i(a) \end{pmatrix},
\end{equation}
where $V_i$ is the Vandermonde matrix in $i+1$ variables.  Using this identity, we show that for $0 \leqslant j \leqslant i$ and $k \geqslant 0$,
\begin{equation}
  \pd_{\theta}^j(a)^{q^k} = \sum_{\ell=0}^i \kappa_{ij\ell} \bigl( t - \theta^{q^k}, t^q - \theta^{q^k}, \dots,
  t^{q^i} - \theta^{q^k} \bigr) a(t)^{q^{\ell}},
\end{equation}
where $\kappa_{ij\ell}$ is the entry of $V_i^{-1}$ in row~$j$ and column~$\ell$ and is completely explicit (see~\eqref{E:kappadef}).  See Proposition~\ref{P:hyperkappa} for more details.  Somewhat surprisingly, as the left-hand side involves only~$\theta$, the expression on the right is independent of $t$, and this leads to the following result on certain hyperderivative power sums.

\begin{theorem:hypersum}
Let $i \geqslant 1$, and let $1 \leqslant s \leqslant q-1$.  Then for any $\mu_1, \dots, \mu_s \geqslant 0$ and any $0 \leqslant j_1, \dots, j_s \leqslant i$,
\begin{align*}
  \sum_{a \in A_{i+}} &\frac{\pd_{\theta}^{j_1}(a)^{q^{\mu_1}} \cdots \pd_{\theta}^{j_s}(a)^{q^{\mu_s}}}{a} \\
  & \hspace*{30pt} = \frac{1}{L_i} \prod_{r=1}^s (-1)^{i-j_r} e_{i,i-j_r} \bigl( \theta - \theta^{q^{\mu_r}}, \theta^q - \theta^{q^{\mu_r}}, \dots, \theta^{q^{i-1}} - \theta^{q^{\mu_r}} \bigr).
\end{align*}
\end{theorem:hypersum}

Here $L_i =(\theta - \theta^q)(\theta-\theta^{q^2}) \cdots (\theta-\theta^{q^i}) \in A$, and $e_{ij}$ denotes the elementary symmetric polynomial in $i$ variables of degree $j$.

As pointed out by the referee, one can apply the techniques of Pellarin and Perkins~\cite{PellarinPerkins20} to obtain formulas of the type in Theorem~\ref{T:hypersum}, without the restriction on~$s$.  We explore the implications of their results to log-algebraic identities in \S\ref{S:Anderson}.

To demonstrate the connections between hyperderivative power sums and Anderson's log-algebraicity results, we recall that the Carlitz module $C$ is the $A$-module structure placed on any $A$-algebra $R$ by setting
\[
  C_{\theta}(x) = \theta x+ x^q, \quad x \in R.
\]
As such for any $a \in A$, the Carlitz operation of $a$ on $R$ is given by a polynomial,
\[
  C_a(x) = \sum_{k=0}^{\deg a} \brac{a}{k} x^{q^k} \in A[x],
\]
where $\brac{a}{0} = a$ and $\brac{a}{\deg a} \in \FF_q$ is the leading coefficient of $a$.  Carlitz~\cite{Carlitz35} showed that
\begin{equation} \label{E:Carlitzbracintro}
  \brac{a}{k} = \sum_{j=0}^k \frac{a^{q^j}}{D_j L_{k-j}^{q^j}},
\end{equation}
where $L_{k}$ is defined in the previous paragraph, and $D_j = (\theta^{q^j}-\theta)(\theta^{q^j}-\theta^q) \cdots (\theta^{q^{j}} - \theta^{q^{j-1}})$.  However, by investigating the interplay of the definition of $\brac{a}{k}$ and hyperderivatives, we prove a new formula for $\brac{a}{k}$ in Proposition~\ref{P:newbrac},
\begin{equation} \label{E:newbracintro}
  \brac{a}{k} = \sum_{j=k}^{\deg a} \pd_{\theta}^j(a) \cdot h_{k,j-k} ( \theta^q-\theta, \theta^{q^2} -\theta, \dots,
  \theta^{q^k}-\theta),
\end{equation}
where $h_{kj}$ represents the complete homogeneous symmetric polynomial in $k$ variables of degree~$j$ (see \S\ref{S:Prelim} for details).

In~\cite{And96}, Anderson proved the following power series identity that he termed a log-algebraicity result.  For new variables $x$ and $z$ and for $\beta(x) \in A[x]$, we set
\begin{equation}
  \cP(\beta,z) \assign \exp_C \left( \sum_{a \in A_+} \frac{\beta(C_a(x))}{a} z^{q^{\deg a}} \right) \in \power{K[x]}{z},
\end{equation}
where $\exp_C(z) = \sum_{i \geqslant 0} z^{q^i}/D_i$ is the Carlitz exponential (see~\S\ref{S:Prelim}).

\begin{theorem:Anderson}[{Anderson~\cite[Thm.~3, Prop.~8]{And96}}]
The power series $\cP(\beta,z) \in \power{K[x]}{z}$ is in fact a polynomial.  Moreover,
\[
\cP(\beta,z) \in A[x,z].
\]
\end{theorem:Anderson}

Anderson then used these identities to show that special values at $s=1$ of Goss $L$-functions $L(\chi,s)$ for Dirichlet characters could be expressed in a direct way as $\oK$-linear combinations of Carlitz logarithms of values of special polynomials,
\[
  P_m(x,z) \assign \cP(x^m,z), \quad m \geqslant 0,
\]
at Carlitz torsion points.  In~\cite[\S 8.10]{Thakur}, Thakur discovered exact identities for Anderson's special polynomials, using Carlitz's formula~\eqref{E:Carlitzbracintro} and formulas for $S_i(k)$ (see Theorem~\ref{T:CarlitzLee}).  His result can be restated in terms of symmetric polynomials as follows.

\begin{theorem:Thakur}[{Thakur~\cite[\S 8.10]{Thakur}}] \
\begin{enumerate}
\item[(a)] Let $m = q^{\mu}$ for $\mu \geqslant 0$.  Then
\[
  P_{q^\mu}(x,z) = \sum_{d=0}^{\mu} (-1)^{\mu-d} \cdot C_{e_{\mu,\mu-d}(\theta, \dots, \theta^{q^{\mu-1}})} \bigl(
  C_{\theta^d}(x)\cdot z \bigr).
\]
\item[(b)] Suppose $m$ is of the form $m = q^{\mu_1} + \cdots + q^{\mu_s}$ for $\mu_j \geqslant 0$ and $1 \leqslant s \leqslant q-1$.  Then
\[
  P_m(x,z) = \sum_{d_1=0}^{\mu_1} \cdots \sum_{d_s=0}^{\mu_s} (-1)^{\sum \mu_j - \sum d_j} \cdot
  C_{\prod_{r=1}^s e_{\mu_r,\mu_r-d_r} (\theta, \dots, \theta^{q^{\mu_r-1}})}
  \Biggl( \Biggl( \prod_{r=1}^s C_{\theta^{d_r}}(x) \Biggr)\cdot z \Biggr).
\]
\end{enumerate}
\end{theorem:Thakur}

In \S\ref{S:Thakur} of the present paper we use~\eqref{E:newbracintro} on relating $\brac{a}{k}$ to hyperderivatives and Theorem~\ref{T:hypersum} on hyperderivative power sums to devise a new proof of Thakur's theorem.  In addition to these ingredients the proof relies heavily on several identities for symmetric polynomials.

The outline of the paper is as follows.  After laying out preliminaries on the Carlitz module, hyperderivatives, and symmetric polynomials in~\S\ref{S:Prelim}, we use these objects to construct formulas for Carlitz multiplication coefficients $\brac{a}{k}$ in~\S\ref{S:Brac}.  In \S\ref{S:Vandermonde} we investigate the connections between hyperderivatives and Vandermonde matrices as well as recall connections with a theorem of Voloch~\cite{Voloch98}.  In \S\ref{S:Sums} we apply the previous techniques to prove formulas for hyperderivative power sums, and then we bring all of these results together in \S\ref{T:Thakur} to give a new proof of Thakur's theorem.  Finally, in \S\ref{S:Anderson} we investigate formulas of Pellarin and Perkins and provide a new proof of Anderson's theorem.

\begin{acknowledgments}
It is with great pleasure that I dedicate this paper to my friend and mentor David Goss.  After discussing with him some years ago part of the material of what would eventually be this paper, David wrote a blog post~\cite{Goss13blog} that outlined the aspects he found most intriguing and included further insights in~\cite{Goss17}.  I would like to thank him for his immense contributions to function field arithmetic and for the interest and enthusiasm he took to my work throughout my career.

I would like to thank the referee, who made several invaluable suggestions that improved the scope of this paper.  In particular, the referee pointed out the connections between Proposition~\ref{P:newbrac} and previous work of Jeong~\cite{Jeong00}, and the referee apprised us of formulas of Pellarin and Perkins~\cite{PellarinPerkins20} that would facilitate using the techniques of this paper to obtain a new proof of Theorem~\ref{T:Anderson}.  I also thank O.~Gezmi\c{s} for helping compare the contents of~\cite{DemeslayPhD}, \cite{PellarinPerkins20}.
\end{acknowledgments}

\section{Preliminaries} \label{S:Prelim}

For $q$ a fixed power of a prime $p$, let $A = \FF_q[\theta]$ be a polynomial ring in one variable over the finite field with $q$ elements, and let $K = \FF_q(\theta)$ be its fraction field.  We take $K_{\infty} = \laurent{\FF_q}{1/\theta}$ for the completion of $K$ at its infinite place, and we take $\KK$ for the completion of an algebraic closure of $K_{\infty}$.  We denote the set of monic elements of $A$ by $A_+$, and for each $i \geqslant 0$, we set
\begin{equation}
  A_{i+} \assign \{ a \in A_+ \mid \deg a = i \}.
\end{equation}
For $i \geqslant 0$, we define the polynomials $[i] = \theta^{q^i} - \theta \in A$, we set $L_0=D_0=1$, and we set
\begin{equation}
D_i \assign [i] [i-1]^{q} \cdots [1]^{q^{i-1}}, \quad L_i \assign (-1)^i [i][i-1] \cdots [1], \quad i \geqslant 1.
\end{equation}

Letting $\tau : \KK \to \KK$ be the $q$-th power Frobenius endomorphism, the \emph{Carlitz module} is the Drinfeld module $C : A \to A[\tau]$ defined by
\[
  C_{\theta} = \theta + \tau.
\]
The ring $A[\tau]$ is the ring of twisted polynomials in $\tau$ with coefficients in $A$.  For $a \in A$ and $k \in \ZZ_+$, we define $\brac{a}{k} \in A$ by
\begin{equation} \label{E:bracdef}
  C_a = \sum_{k=0}^{\deg a} \brac{a}{k} \tau^k.
\end{equation}
As usual $C$ defines an $A$-module structure on any $A$-algebra by way of the commutative polynomials for $a \in A$,
\[
  C_a(x) = \sum_{k=0}^{\deg a} \brac{a}{k} x^{q^k} \in A[x].
\]
The Carlitz exponential and logarithm are defined by the infinite series,
\[
  \exp_C(z) = \sum_{i \geqslant 0} \frac{z^{q^i}}{D_i}, \quad
  \log_C(z) = \sum_{i \geqslant 0} \frac{z^{q^i}}{L_i}.
\]
They are mutual inverses of each other and for each $a \in A$, we have $\exp_C(az) = C_a(\exp_C(z))$.  For more information about the Carlitz module, and Drinfeld modules in general, the reader is directed to~\cite[Ch.~3--4]{Goss}, \cite[Ch.~2--3]{Thakur}.

For a field $F$ and a variable $\theta$ transcendental over $F$, the \emph{hyperdifferential operators with respect to $\theta$}, $\pd_\theta^{j}: F[\theta] \to F[\theta]$, $j \geqslant 0$, are defined $F$-linearly by setting $\pd_{\theta}^{j}(\theta^n) = \binom{n}{j} \theta^{n-j}$.  We note that $\binom{n}{j} = 0$ when $n < j$, and so these maps are well-defined.  Hyperderivatives then extend uniquely to operators $\pd_\theta^j : F(\theta)_v^{\sep} \to F(\theta)_v^{\sep}$ on the separable closure of the completion of $F(\theta)$ at a place $v$ (see~\cite[\S 4]{Conrad00}, \cite[\S 2]{Jeong11}).  Hyperderivatives satisfy several kinds of differentiation rules, such as the product rule,
\begin{equation}
  \pd_{\theta}^j (fg) = \sum_{k=0}^j \pd_{\theta}^k(f) \pd_{\theta}^{j-k}(f), \quad f, g \in F(\theta)_v^{\sep},
\end{equation}
and the composition rule,
\begin{equation}
  \pd_{\theta}^j(\pd_{\theta}^k(f)) = \pd_{\theta}^k (\pd_\theta^j(f)) = \binom{j+k}{j} \pd_{\theta}^{j+k}(f), \quad
  f \in F(\theta)_v^{\sep}.
\end{equation}
For various versions of the product rule, quotient rule, power rule, and chain rule, the reader is directed to~\cite[\S 2]{Jeong11}, \cite[\S 2.3]{PLogAlg}.

For a sequence of independent variables $\theta_1, \dots, \theta_m$, we can define compatible partial hyperderivatives,
\[
  \pd_{\theta_i}^j : F(\theta_1, \dots, \theta_m) \to F(\theta_1, \dots, \theta_m),
\]
in the natural way with the property that for $i \neq i'$ we have $\pd_{\theta_{i}}^{j} \circ \pd_{\theta_{i'}}^{j'} = \pd_{\theta_{i'}}^{j'} \circ \pd_{\theta_i}^j$ (see \cite[Ch.~2]{Okugawa}).  Mostly we will focus on the case of two variables, say $\theta$ and $t$, and for functions $f \in F(\theta,t)$, we say that $f$ is regular at $t=\theta$ if $f|_{t=\theta} \assign f(\theta,\theta)$ is well-defined in $F(\theta)$.  For $f$ regular at $t=\theta$, it follows from the quotient rule that $\pd_{t}^j(f)$ is also regular at $t=\theta$ for $j \geqslant 0$.  The following standard Taylor series lemma will be used throughout.

\begin{lemma} \label{L:Taylor}
For a field $F$, let $f \in F(\theta,t)$ be regular at $t=\theta$.  Then as an element of $\power{F(\theta)}{t-\theta}$,
\[
  f(\theta,t) = \sum_{j=0}^{\infty} \pd_t^j(f)\big|_{t=\theta} \cdot (t-\theta)^j.
\]
\end{lemma}

We further recall definitions of symmetric polynomials.  For independent variables $t$, $x_1, x_2, \ldots$ over $\ZZ$, the \emph{elementary symmetric polynomials} $e_{ij} \in \ZZ[x_1, \dots, x_i]$ are defined by
\begin{equation} \label{E:eijdef}
\sum_{j=0}^{i} e_{ij}(x_1, \dots, x_i) t^j = (1 + x_1 t)(1 + x_2 t) \cdots (1 + x_i t),
\end{equation}
and the \emph{complete homogeneous symmetric polynomials} $h_{ij} \in \ZZ[x_1, \dots, x_i]$ are defined by
\begin{equation} \label{E:hijdef}
\sum_{j=0}^{\infty} h_{ij}(x_1, \dots, x_i)t^j = \frac{1}{(1 - x_1 t)(1 - x_2 t) \cdots (1 - x_i t)}.
\end{equation}
It is readily apparent that $e_{i0}=h_{i0} = 1$ and $e_{i1}=h_{i1} = x_1 + \dots + x_i$ for all $i \geqslant 0$.  The polynomial $h_{ij}$ consists of the sum of all monomials in $x_1, \dots, x_i$ of total degree $j$, so for example,
\[
  h_{22} = x_1^2 + x_1 x_2 + x_2^2, \quad h_{23} = x_1^3 + x_1^2 x_2 + x_1 x_2^2 + x_2^3.
\]
By convention we extend $e_{ij}$ and $h_{ij}$ to all $j \in \ZZ$ by setting
\begin{gather*}
  \textup{($j < 0$ or $j > i$)} \quad \Rightarrow \quad e_{ij} = 0, \\
  j < 0 \quad \Rightarrow \quad h_{ij} = 0.
\end{gather*}
In this notation $e_{ij}$ (resp.~$h_{ij}$) represents the elementary symmetric polynomial (resp.~complete homogenous symmetric polynomial) of degree $j$ in $i$ variables.  For more detailed information on symmetric polynomials, see~\cite[Ch.~7]{Stanley}.

The polynomials $e_{ij}$ and $h_{ij}$ satisfy several standard recurrence relations through relations on their generating functions.  The first is the pair of recursions,
\begin{align} \label{E:eijrec1}
  e_{ij}(x_1, \dots, x_i) &= e_{i-1,j}(x_1, \dots, \widehat{x_{\ell}}, \dots, x_i) + x_{\ell} e_{i-1,j-1}
    (x_1, \dots, \widehat{x_{\ell}}, \dots, x_i) \\
  \label{E:hijrec1}
  h_{ij}(x_1, \dots, x_i) &= h_{i-1,j}(x_1, \dots, \widehat{x_{\ell}}, \dots, x_i) + x_\ell h_{i,j-1}
    (x_1, \dots, x_i),
\end{align}
where `$\widehat{x_{\ell}}$' indicates that the variable $x_{\ell}$ is omitted.  These imply for $i \geqslant 1$ and $j \geqslant 0$,
\begin{align} \label{E:eijspec0}
  e_{ij}(x_1, \dots, x_i)|_{x_{\ell}=0} = e_{i-1,j}(x_1, \dots, \widehat{x_{\ell}}, \dots, x_{i}), \\
  \label{E:hijspec0}
  h_{ij}(x_1, \dots, x_i)|_{x_{\ell}=0} = h_{i-1,j}(x_1, \dots, \widehat{x_{\ell}}, \dots, x_{i}).
\end{align}
A second type of recursive relation holds for $e_{ij}$ and $h_{ij}$ in tandem.  See also \cite[Pf.~of Thm.~3.2]{OrucAkmaz04} and~\cite[\S 7.6]{Stanley}.

\begin{proposition} \label{P:symmrec2}
For fixed $i \geqslant 1$ and $1 \leqslant k \leqslant i$,
\[
  \sum_{j=k}^i (-1)^{i-j} e_{i-1,i-j}(x_1, \dots, x_{i-1}) \cdot h_{k,j-k}(x_1, \dots, x_k) =
  \begin{cases}
  1 & \textup{if $k=i$,} \\
  0 & \textup{if $k < i$.}
  \end{cases}
\]
\end{proposition}

\begin{proof}
One easily checks the case $k=i$ from the definitions of $e_{ij}$ and $h_{ij}$.  Thus we can now assume that $k \leqslant i-1$.  From~\eqref{E:eijdef} and~\eqref{E:hijdef}, we find
\[
  \sum_{j=0}^{i-1} \sum_{\ell = 0}^{\infty} (-1)^j e_{i-1,j}(x_1, \dots, x_{i-1}) \cdot h_{k \ell}(x_1, \dots, x_k) t^{j+\ell} = (1 - x_{k+1} t) \cdots (1 - x_{i-1} t),
\]
and by reversing the order of the outer sum we have
\[
 \sum_{j=1}^{i} \sum_{\ell = 0}^{\infty} (-1)^{i-j} e_{i-1,i-j}(x_1, \dots, x_{i-1}) \cdot h_{k \ell}(x_1, \dots, x_k) t^{i-j+\ell} = (1 - x_{k+1} t) \cdots (1 - x_{i-1} t).
\]
The right-hand side has degree $i-k-1$ in $t$, and so the coefficient of $t^{i-k}$ on the right is~$0$.  On the other hand, $t^{i-k}$ appears in the left-hand side precisely when $\ell = j-k$, which implies
\[
  \sum_{j=1}^i (-1)^{i-j} e_{i-1,i-j}(x_1, \dots, x_{i-1}) \cdot h_{k,j-k}(x_1, \dots, x_k) = 0.
\]
Since again $h_{k,j-k} = 0$ when $j < k$, the result follows.
\end{proof}

\begin{remark} \label{R:EdHd}
One useful way of expressing Proposition~\ref{P:symmrec2} is the following.  For $d \geqslant 1$, define lower triangular $d \times d$ matrices with entries in $\ZZ[x_1, \dots, x_{d-1}]$,
\begin{align*}
\cE_{d} &= \begin{pmatrix}
1 & 0 & 0 & \cdots & 0\\
-e_{11} & 1 & 0 & \cdots & 0 \\
e_{22} & -e_{21} & 1 & \cdots & 0\\
\vdots & \vdots & \ddots & \ddots & \vdots  \\
(-1)^{d-1} e_{d-1,d-1} & (-1)^{d-2} e_{d-1,d-2} & \cdots & -e_{d-1,1} & 1
\end{pmatrix}, \\[10pt]
\cH_d &= \begin{pmatrix}
1 & 0 & 0 & \cdots & 0 \\
h_{11} & 1 & 0 & \cdots & 0 \\
h_{12} & h_{21} & 1 & \cdots & 0 \\
\vdots & \vdots & \ddots & \ddots & \vdots  \\
h_{1,d-1} & h_{2,d-2} & \cdots & h_{d-1,1} & 1
\end{pmatrix}.
\end{align*}
Then Proposition~\ref{P:symmrec2} is equivalent to the identity, for $i \leqslant d$,
\[
  \cE_d \cdot \cH_d = I_d,
\]
where $I_d$ is the $d\times d$ identity matrix.  The product $\cH_d \cdot \cE_d = I_d$ produces a companion formula for Proposition~\ref{P:symmrec2}, which we do not state here but arises in the proof of Proposition~\ref{P:bracthetam}.

These matrices also appear in calculations of the $LDU$ decomposition of the Vandermonde matrix in $x_1, \dots, x_i$ (see~\cite{OrucAkmaz04}, \cite{OrucPhillips00}).  In the next section we will use them to derive new formulas for the coefficients $\brac{a}{k}$ of the Carlitz multiplication polynomials~$C_a$ from~\eqref{E:bracdef}.
\end{remark}

\section{Carlitz multiplication coefficients} \label{S:Brac}
Let $M = \KK[t]$, where $t$ is a variable independent from $\theta \in \KK$.  Then we can define a left $\KK[\tau]$-module structure on $M$ by setting $\tau m \assign (t-\theta) m^{(1)}$ for $m \in M$.  As a left $\KK[t,\tau]$-module, $M$ then has the structure of an Anderson $t$-motive, in the sense of~\cite{And86}, which is isomorphic to the $t$-motive of the Carlitz module (see~\cite[\S 4.3]{BPrapid}, \cite[\S 5.8]{Goss}).  Now define polynomials $\mu_{k}(t) \in A[t]$ for $k \geqslant 0$ by setting $\mu_0=1$, and
\begin{equation} \label{E:mudef}
  \mu_k(t) \assign (t-\theta)(t-\theta^q) \cdots (t-\theta^{q^{k-1}}), \quad k \geqslant 1.
\end{equation}
As $\deg_t \mu_k = k$, we see that $\{ \mu_k \}$ forms a $\KK$-basis for $\KK[t]$.
Then the following proposition is due to Thakur, based on previous work of Drinfeld and Mumford~\cite{Mumford78}.

\begin{proposition}[{Thakur~\cite[\S 0.3.5]{Thakur93}}] \label{P:bracThakur}
For $a = a(\theta) \in A$, the expansion of $a(t) \in \KK[t]$ in terms of the basis $\{ \mu_k\}$ is given by
\[
  a(t) = \sum_{k=0}^{\deg a} \brac{a}{k} \mu_k(t).
\]
\end{proposition}

That is, the Carlitz multiplication coefficients of $C_a$ from~\eqref{E:bracdef} are also the coefficients of $a(t)$ in terms of the our $\KK$-basis on the $t$-motive $M$.  See \cite[\S 7.11]{Goss}, \cite[\S 3]{GreenP18}, \cite[\S 0.3]{Thakur93}, for additional discussion and details.

In this section we will use Proposition~\ref{P:bracThakur} to derive a new formula for $\brac{a}{k}$ in terms of hyperderivatives and symmetric polynomials.  We first observe from~\eqref{E:eijdef} that for $k \geqslant 1$,
\begin{align} \label{E:mukdec1}
  \mu_k(t) &= (t - \theta)(t - \theta - [1]) \cdots (t-\theta-[k-1]) \\
  &= \sum_{j=1}^k (-1)^{k-j} e_{k-1,k-j}([1], \dots, [k-1]) (t-\theta)^j. \notag
\end{align}
For $d \geqslant 0$, we take $\cN_d \in \Mat_{d+1} ( \ZZ[x_1, \dots, x_{d-1}])$ to be
\[
  \cN_d = \cN_d(x_1, \dots, x_{d-1}) = \begin{pmatrix} 1 & 0 \\ 0 & \cE_d \end{pmatrix},
\]
where $\cE_d$ is defined in Remark~\ref{R:EdHd}, and then~\eqref{E:mukdec1} implies
\begin{equation} \label{E:mukdec2}
\begin{pmatrix} \mu_0(t) \\ \mu_1(t) \\ \mu_2(t) \\ \vdots \\ \mu_d(t) \end{pmatrix}
= \cN_d([1], \dots, [d-1])
\begin{pmatrix} 1 \\ t-\theta \\ (t-\theta)^2 \\ \vdots \\ (t-\theta)^d \end{pmatrix}.
\end{equation}
Now let $a \in A$ have degree $d$ in $\theta$.  From Lemma~\ref{L:Taylor}, it follows that
\[
  a(t) = \sum_{j=0}^d \pd_{\theta}^j(a) (t-\theta)^j.
\]
Therefore, \eqref{E:mukdec2} implies that
\begin{align} \label{E:aNddec}
a(t) &= \bigl( \pd_{\theta}^0(a), \dots, \pd_{\theta}^d(a) \bigr)
  \begin{pmatrix} 1 \\ t-\theta \\ \vdots \\ (t-\theta)^d \end{pmatrix} \\
  &= \bigl( \pd_{\theta}^0(a), \dots, \pd_{\theta}^d(a) \bigr) \cN_d([1], \dots, [d-1])^{-1}
  \begin{pmatrix} \mu_0(t) \\ \mu_1(t) \\ \vdots \\ \mu_d(t) \end{pmatrix}. \notag
\end{align}
Now by Remark~\ref{R:EdHd}, we see that
\[
  \cN_d^{-1} = \begin{pmatrix} 1 & 0 \\ 0 & \cH_d \end{pmatrix},
\]
and by comparing entries of $\cH_d$ with Proposition~\ref{P:bracThakur} and~\eqref{E:aNddec}, we have proved the following proposition.

\begin{proposition} \label{P:newbrac}
Let $a \in A$ have degree $d \geqslant 0$.  Then for $0 \leqslant k \leqslant d$,
\[
  \brac{a}{k} = \sum_{j=k}^d \pd_{\theta}^j(a) \cdot h_{k,j-k} ([1], \dots, [k]).
\]
\end{proposition}

\begin{remark}
The proof given above works only for $k \geqslant 1$, but the formula in the proposition is valid also when $k=0$.  In this case $\brac{a}{0} = a$, whereas the definition of $h_{ij}$ implies that $h_{0j} = 1$ if $j=0$ and that $h_{0j}=0$ if $j > 0$, and the formula holds.
\end{remark}

\begin{remark}
It is worth comparing the formula in Proposition~\ref{P:newbrac} with other formulas for $\brac{a}{k}$.  For example, there is the formula due to Carlitz~\cite[Thm.~2.1]{Carlitz35} that
\begin{equation} \label{E:EbracCarlitz}
  E_k(z) \assign \sum_{j=0}^k \frac{z^{q^j}}{D_j L_{k-j}^{q^j}} \in K[z] \quad \Rightarrow
  \quad \brac{a}{k} = E_k(a).
\end{equation}
By its definition as a coefficient of the polynomial $C_a \in A[\tau]$, we know that $\brac{a}{k} \in A$, but in contrast to Proposition~\ref{P:newbrac} this is not particularly clear from Carlitz's formula on its own.  Jeong~\cite{Jeong00} obtained additional formulas for $\brac{a}{k}$ through hyperdifferential identities involving $E_k(z)$.  See Proposition~\ref{P:Jeong} and the subsequent discussion for connections with Proposition~\ref{P:newbrac}.

In \cite[\S 8.10]{Thakur}, Thakur uses \eqref{E:EbracCarlitz} to give explicit formulas for special polynomials from Anderson's log-algebraicity theorem for the Carlitz module~\cite[Prop.~8]{And96}.  In \S\ref{S:Thakur}, we will reformulate Thakur's results (see Theorem~\ref{T:Thakur}) and use Proposition~\ref{P:newbrac} to design a new proof.
\end{remark}

We can also write $\brac{a}{k}$ in terms of symmetric polynomials in other ways that we will need.  The following proposition provides a different formula for $\brac{\theta^m}{k}$.

\begin{proposition} \label{P:bracthetam}
Let $m \geqslant 0$ and $0 \leqslant k \leqslant m$.  Then
\[
\brac{\theta^m}{k} = h_{k+1,m-k}\bigl( \theta, \theta^q, \dots, \theta^{q^k} \bigr).
\]
\end{proposition}

\begin{proof}
The essential argument is to expand $t^m$ in terms of the polynomials $\mu_k(t)$ and use Proposition~\ref{P:bracThakur}.  To do this we consider
\begin{align*}
  \sum_{k=0}^m h_{k+1,m-k} \bigl( \theta, \dots, \theta^{q^k} \bigr) \mu_k(t)
  &= \sum_{k=0}^m h_{k+1,m-k} \bigl( \theta, \dots, \theta^{q^k} \bigr) \bigl( t-\theta \bigr) \cdots \bigl( t-\theta^{q^{k-1}} \bigr) \\
  &= \begin{aligned}[t]
  \sum_{k=0}^m\, &h_{k+1,m-k} \bigl( \theta, \dots, \theta^{q^k} \bigr) \\
    &{} \times \sum_{j=0}^k (-1)^{k-j} e_{k,k-j} \bigl(\theta, \dots, \theta^{q^{k-1}}\bigr) t^j
  \end{aligned}
  \\
  &= \sum_{j=0}^m t^j \sum_{k=j}^m (-1)^{k-j} h_{k+1,m-k} \cdot e_{k,k-j}.
\end{align*}
Now the inner sum is the same as the entry in row $m+1$ and column $j+1$ in the matrix product $\cH_d \cdot \cE_d = I_d$ from Remark~\ref{R:EdHd} (with any $d > m$).  Thus the inner sum is $1$ if $j=m$ and $0$ otherwise, so
\[
  \sum_{k=0}^m h_{k+1,m-k} \bigl( \theta, \dots, \theta^{q^k} \bigr) \mu_k(t)
  = h_{m+1,0}\bigl( \theta, \dots, \theta^{q^m} \bigr) \cdot e_{m,0} \bigl( \theta, \dots, \theta^{q^{m-1}} \bigr) \cdot t^m = t^m,
\]
and the result follows from Proposition~\ref{P:bracThakur}.
\end{proof}

As was pointed out by the referee, Proposition~\ref{P:newbrac} is similar to work of Jeong~\cite[Cor.~1]{Jeong00}, and below we give another proof of Proposition~\ref{P:newbrac} that follows from a combination of Jeong's results and Proposition~\ref{P:bracthetam}.  The proof below requires also Lemma~\ref{L:ehdiff} from later in the paper, whose proof is independent of any of these considerations.  However, for the sake of exposition we have left Lemma~\ref{L:ehdiff} where it was in the original version of this paper, and we refer the reader to its proof in \S\ref{S:Thakur}.  Recalling $E_k(z)$ from~\eqref{E:EbracCarlitz}, Jeong proved the following.

\begin{proposition}[{Jeong~\cite[Cor.~1]{Jeong00}}] \label{P:Jeong}
For $j$, $k \geqslant 0$, set $B_{j,k} \assign \sum_{i=0}^{j} (-1)^{j-i} \pd_{\theta}^i(\theta^j) E_k(\theta^i)$.  Then for $u \in \power{\FF_q}{\theta}$ we have
\[
  E_k(u) = \sum_{j=0}^\infty B_{j,k} \pd_{\theta}^j(u).
\]
\end{proposition}

\begin{proof}[Second proof of Proposition~\ref{P:newbrac}]
Because $\brac{a}{k} = E_k(a)$ by~\eqref{E:EbracCarlitz} and since $h_{ij} = 0$ for $j < 0$, it suffices by Proposition~\ref{P:Jeong} to show that $B_{j,k} = h_{k,j-k}([1], \dots, [k])$ for $k \leqslant j \leqslant d$.  Combining the definition of $B_{j,k}$ with \eqref{E:EbracCarlitz}, we see
\[
  B_{j,k} = \sum_{i=k}^j (-1)^{j-i} \binom{j}{i} \theta^{j-i} \brac{\theta^i}{k}
  = \sum_{i=k}^j (-1)^{j-i} \binom{j}{i} \theta^{j-i} h_{k+1,i-k} \bigl( \theta, \theta^q, \dots, \theta^{q^k} \bigr),
\]
where the second equality follows from Proposition~\ref{P:bracthetam}.  First, taking $i \leftarrow i+k$ to reindex the sum and then using $d \leftarrow k+1$, $j \leftarrow i$, $k \leftarrow j-k$ in Lemma~\ref{L:ehdiff}(b), we find
\begin{align*}
  B_{j,k} &= \sum_{i=0}^{j-k} (-1)^{j-i-k} \binom{j}{i+k} \theta^{j-i-k} h_{k+1,i} \bigl( \theta, \theta^q, \dots, \theta^{q^k} \bigr) \\
  &= (-1)^{j-k} h_{k+1,j-k} \bigl(T-\theta, \dots, T-\theta^{q^k} \bigr)\big|_{T=\theta}.
\end{align*}
Finally from~\eqref{E:hijspec0} we see that $B_{j,k} = h_{k,j-k}([1], \dots, [k])$.
\end{proof}

\begin{remark}
It is possible to extend Propositions~\ref{P:newbrac} and~\ref{P:bracthetam} to multiplication coefficients of tensor powers of the Carlitz module, and the reader is directed to~\cite[Ch.~3]{PLogAlg} for further details.
\end{remark}

\section{Vandermonde matrices and a theorem of Voloch} \label{S:Vandermonde}

In~\cite{Voloch98}, Voloch proved the following proposition that relates $q$-th powers of power series over $\FF_q$ to their hyperderivatives.  Voloch's proof is straightforward, but below we give another proof that then lends itself well to generalizations for our purposes.

\begin{proposition}[{Voloch~\cite{Voloch98}}] \label{P:Voloch}
Let $g \in \power{\FF_q}{\theta}$.  Then for $k \geqslant 0$,
\[
  g^{q^k} = \sum_{j=0}^{\infty} \pd_{\theta}^j(g) \cdot [k]^j,
\]
where by convention we set $[0]^0 = 1$.
\end{proposition}

\begin{proof}
By substituting $\theta \leftarrow t$ into $g$, the definition of hyperderivatives yields
\[
  g(t) = \sum_{j=0}^{\infty} \pd_{\theta}^j(g) (t-\theta)^j.
\]
Then for $k \geqslant 0$,
\[
  g(t)^{q^k} = g\bigl( t^{q^k} \bigr)  = \sum_{j=0}^{\infty} \pd_{\theta}^j(g) (t^{q^k} - \theta)^j.
\]
Upon substituting $t \leftarrow \theta$, we obtain the result.
\end{proof}

Fixing $i \geqslant 0$, for variables $x_0, \dots, x_i$ we define the $(i+1) \times (i+1)$ Vandermonde matrix,
\begin{equation}
  V_i(x_0, \dots, x_i) = \begin{pmatrix}
  1 & x_0 & \cdots & x_0^i \\
  1 & x_1 & \cdots & x_1^i \\
  \vdots & \vdots & & \vdots \\
  1 & x_i & \cdots & x_i^i
  \end{pmatrix}.
\end{equation}
For a polynomial $a \in A$ of degree at most $i$, as in the proof of Proposition~\ref{P:Voloch}, we have identities for $k \geqslant 0$,
\[
  a(t)^{q^k} = a\bigl( t^{q^k} \bigr) = \sum_{j=0}^i \pd_{\theta}^j(a) \bigl(t^{q^k} - \theta \bigr)^j.
\]
This yields a matrix identity
\begin{equation}
\begin{pmatrix} a(t) \\ a(t)^q \\ \vdots \\ a(t)^{q^i} \end{pmatrix}
= V_i \bigl( t-\theta, t^q - \theta, \dots, t^{q^i} -\theta \bigr)
\begin{pmatrix} a \\ \pd_{\theta}^1(a) \\ \vdots \\ \pd_{\theta}^i(a) \end{pmatrix}.
\end{equation}
Inverting the Vandermonde matrix, we have
\begin{equation}
\begin{pmatrix} a \\ \pd_{\theta}^1(a) \\ \vdots \\ \pd_{\theta}^i(a) \end{pmatrix}
= V_i \bigl( t-\theta, t^q - \theta, \dots, t^{q^i} -\theta \bigr)^{-1}
\begin{pmatrix} a(t) \\ a(t)^q \\ \vdots \\ a(t)^{q^i} \end{pmatrix}.
\end{equation}
Notably the left-hand side involves only the variable $\theta$, implying that the expression on the right-hand side is independent of $t$.  The inverse of $V_i$ can be obtained through various means (e.g., see \cite[p.~36]{Knuth}, \cite[Thm.~3.2]{OrucAkmaz04}), and setting
\[
  V_i(x_0, \dots, x_i)^{-1} = \bigl( \kappa_{ij\ell}(x_0, \dots, x_i) \bigr)_{0 \leqslant j \leqslant i,\, 0 \leqslant \ell \leqslant i},
\]
the entries $\kappa_{ij\ell}$ can  be expressed in terms of elementary symmetric functions,
\begin{equation} \label{E:kappadef}
\kappa_{ij\ell}(x_0, \dots, x_i) = \frac{(-1)^{i-j} e_{i,i-j}(x_0, \dots, \widehat{x_{\ell}}, \dots, x_i)}{\prod_{m \neq \ell} (x_{\ell} - x_m)}.
\end{equation}
The following proposition is then immediate.

\begin{proposition} \label{P:hyperkappa}
For $i \geqslant 0$, let $a \in A$ satisfy $\deg a \leqslant i$.
\begin{enumerate}
\item[(a)] For $0 \leqslant j \leqslant i$,
\[
  \pd_{\theta}^j(a) = \sum_{\ell=0}^{i} \kappa_{ij\ell} \bigl( t-\theta, t^q-\theta, \dots, t^{q^i} - \theta \bigr) a(t)^{q^{\ell}}.
\]
\item[(b)] More generally, for $0 \leqslant j \leqslant i$ and $k \geqslant 0$,
\[
  \pd_{\theta}^j(a)^{q^k} = \sum_{\ell=0}^{i} \kappa_{ij\ell} \bigl( t-\theta^{q^k}, t^q-\theta^{q^k}, \dots, t^{q^i} - \theta^{q^k} \bigr) a(t)^{q^{\ell}}.
\]
\end{enumerate}
\end{proposition}

\begin{proof}
The proof of (a) has already been shown.  For (b), we apply a Frobenius twist.  That is, we need only observe that $\pd_{\theta}^j(a)^{q^k}$, as it is a polynomial in $\FF_q[\theta]$, is obtained from $\pd_{\theta}^j(a)$ by replacing $\theta \leftarrow \theta^{q^k}$.  Thus making the same replacement on the right-hand side of (a) will yield the desired result.
\end{proof}

\begin{remark}
Of note in part (b) of the proposition is that the degree in $t$ on the right-hand side is bounded independently from $k$.
\end{remark}

\begin{remark} \label{R:kappasubbed}
We see from \eqref{E:kappadef} that
\begin{multline}
\kappa_{ij\ell} \bigl( t-\theta^{q^k}, t^q-\theta^{q^k}, \dots, t^{q^i} - \theta^{q^k} \bigr) \\
= \frac{(-1)^{i-j} e_{i,i-j}\bigl( t-\theta^{q^k}, \dots, t^{q^{\ell-1}} - \theta^{q^k}, t^{q^{\ell+1}} - \theta^{q^k}, \dots, t^{q^i}-\theta^{q^k} \bigr)}
{ \bigl( t^{q^\ell} - t\bigr) \cdots \bigl( t^{q^\ell} - t^{q^{\ell-1}} \bigr) \bigl( t^{q^\ell} - t^{q^{\ell+1}} \bigr) \cdots \bigl( t^{q^\ell} - t^{q^i} \bigr)},
\end{multline}
which will be important for calculations in \S\ref{S:Sums}--\ref{S:Thakur}.
\end{remark}

\section{Hyperderivative power sums} \label{S:Sums}

Power sums of polynomials in $A$ are defined for $i \geqslant 0$ and $k \in \ZZ$ by setting
\begin{equation}
  S_i(k) = \sum_{a \in A_{i+}} a^k.
\end{equation}
Research on vanishing criteria and precise formulas for these sums goes back at least to Carlitz~\cite{Carlitz35}, \cite{Carlitz37}, and Lee~\cite{Lee43}, and it has been continued by several authors, \cite{AnglesPellarin14}, \cite{Gekeler88a}, \cite{Goss78}, \cite{Perkins14c}, \cite[Ch.~5]{Thakur}, \cite{Thakur09}.  More generally, one can ask for similar results about \emph{hyperderivative power sums} of the form
\[
  H_i(j_1, \dots, j_\ell; k_1, \dots, k_{\ell}) = \sum_{a \in A_{i+}} \pd_{\theta}^{j_1}(a)^{k_1} \cdots \pd_{\theta}^{j_\ell}(a)^{k_{\ell}},
\]
for $j_1, \dots, j_\ell \geqslant 0$ and $k_1, \dots, k_{\ell} \in \ZZ$.  We will explore some results for reasonably simple hyperderivative power sums in this section, but they are studied in more depth in~\cite[Ch.~5, 8--9]{PLogAlg}.

In this section the following results on $S_i(k)$ for $k \geqslant 0$ are fundamental.  For $k \geqslant 0$, if we write $k = \sum_{j=0}^s k_j q^j$ with $0 \leqslant k_j \leqslant q-1$, then we let
\[
  \sigma_q(k) \assign k_0 + \dots + k_s
\]
be the sum of its digits base $q$.

\begin{theorem}[{Carlitz~\cite[Thm.~9.5]{Carlitz35}, \cite[p.~497]{Carlitz37}; Lee~\cite[Lem.~7.1]{Lee43}; see also Gekeler~\cite[Cor.~2.12]{Gekeler88a} and Thakur~\cite[Thm.~5.1.2]{Thakur}}] \label{T:Sivanish}
Let $i \geqslant 0$.  For $k \geqslant 0$, if $\sigma_q(k) < i(q-1)$, then $S_i(k) = 0$.  Moreover, if $k< q^i -1$, then $S_i(k) = 0$.
\end{theorem}

\begin{theorem}[{Carlitz~\cite[p.~941]{Carlitz39}; Lee~\cite[Thm.~4.1]{Lee43}; see also Gekeler~\cite[Thm.~4.1, Rmk.~6.6]{Gekeler88a}}] \label{T:CarlitzLee}
Let $i \geqslant 0$.  Suppose $1 \leqslant s \leqslant q-1$ and $k = q^{\ell_1} + \cdots + q^{\ell_s} - 1 \geqslant 0$.
\begin{enumerate}
\item[(a)] If $\ell_r < i$ for some $r=1, \dots, s$, then $S_i(k)=0$.
\item[(b)] If $\ell_r \geqslant i$ for all $r = 1, \dots, s$, then
\[
  S_i(k) = \frac{1}{L_i} \prod_{r=1}^s \frac{D_{\ell_r}}{D_{\ell_r-i}^{q^i}}.
\]
\item[(c)] In particular,
\[
  S_i(q^{\ell+i}-1) = \frac{D_{\ell+i}}{L_i D_{\ell}^{q^i}}.
\]
\end{enumerate}
\end{theorem}

The second theorem can be proved especially cleanly by way of results of Angl\`{e}s and Pellarin~\cite{AnglesPellarin14}, who adapted techniques of Simon~\cite[Lem.~4]{AnglesPellarin14}.  In particular, Theorem~\ref{T:CarlitzLee} can be obtained by specializing the following result at $t_1 = \theta^{q^{\ell_1}}, \dots, t_s = \theta^{q^{\ell_s}}$.

\begin{proposition}[{Angl\`{e}s-Pellarin~\cite[Prop.~10]{AnglesPellarin14}}] \label{P:AnglesPellarin}
Let $i \geqslant 0$.  For $0 \leqslant s \leqslant q-1$,
\[
  \sum_{a \in A_{i+}} \frac{a(t_1) \cdots a(t_s)}{a} = \frac{1}{L_i} \prod_{r=1}^s \prod_{\nu=0}^{i-1}
  \bigl( t_r - \theta^{q^{\nu}} \bigr).
\]
\end{proposition}

In this proposition the coefficient of the top degree term in $t_1, \dots, t_s$ is $S_i(-1) = \sum_{a \in A_{i+}} 1/a = 1/L_i$ (see~\cite[Eq.~(9.09)]{Carlitz35}).  The main result of this section is the following.

\begin{theorem} \label{T:hypersum}
Let $i \geqslant 1$, and let $1 \leqslant s \leqslant q-1$.  Then for any $\mu_1, \dots, \mu_s \geqslant 0$ and any $0 \leqslant j_1, \dots, j_s \leqslant i$,
\begin{align*}
  \sum_{a \in A_{i+}} &\frac{\pd_{\theta}^{j_1}(a)^{q^{\mu_1}} \cdots \pd_{\theta}^{j_s}(a)^{q^{\mu_s}}}{a} \\
  & \hspace*{30pt} = \frac{1}{L_i} \prod_{r=1}^s (-1)^{i-j_r} e_{i,i-j_r} \bigl( \theta - \theta^{q^{\mu_r}}, \theta^q - \theta^{q^{\mu_r}}, \dots, \theta^{q^{i-1}} - \theta^{q^{\mu_r}} \bigr).
\end{align*}
\end{theorem}

\begin{proof}
By Proposition~\ref{P:hyperkappa}, we see that
\[
  \sum_{a \in A_{i+}} \frac{\pd_{\theta}^{j_1}(a)^{q^{\mu_1}} \cdots \pd_{\theta}^{j_s}(a)^{q^{\mu_s}}}{a}
  = \sum_{a \in A_{i+}} \frac{1}{a}\cdot \prod_{r=1}^{s} \sum_{\ell_r = 0}^i \kappa_{i j_r \ell_r} \bigl( t - \theta^{q^{\mu_r}}, \dots, t^{q^i} - \theta^{q^{\mu_r}} \bigr)a(t)^{q^{\ell_r}}.
\]
If we substitute $t = \theta$, then we have
\begin{align*}
  \sum_{a \in A_{i+}} \frac{\pd_{\theta}^{j_1}(a)^{q^{\mu_1}} \cdots \pd_{\theta}^{j_s}(a)^{q^{\mu_s}}}{a}
  &= \sum_{a \in A_{i+}} \frac{1}{a}\cdot \prod_{r=1}^{s} \sum_{\ell_r = 0}^i \kappa_{ij_r \ell_r} \bigl( \theta - \theta^{q^{\mu_r}}, \dots, \theta^{q^i} - \theta^{q^{\mu_r}} \bigr)a^{q^{\ell_r}} \\
  &= \begin{aligned}[t]
  \sum_{\ell_1 = 0}^i \cdots \sum_{\ell_s=0}^i \Biggl( \prod_{r=1}^s \kappa_{ij_r \ell_r} &\bigl( \theta - \theta^{q^{\mu_r}}, \dots, \theta^{q^i} - \theta^{q^{\mu_r}} \bigr) \Biggr) \\
   &{}\times \sum_{a \in A_{i+}} a^{q^{\ell_1} + \dots + q^{\ell_s}-1}.
  \end{aligned}
\end{align*}
Thus for fixed $\ell_1, \dots, \ell_s$, each term is multiplied by $S_i(q^{\ell_1} + \dots + q^{\ell_s} - 1)$.  By Theorem~\ref{T:CarlitzLee}, this sum is $0$, unless $\ell_1 = \cdots = \ell_s = i$, in which case
\[
  S_i(q^{\ell_1} + \dots + q^{\ell_s} - 1) = S_i(sq^i - 1) = \frac{D_i^s}{L_i}.
\]
Therefore,
\begin{equation}  \label{E:hypersuminter}
  \sum_{a \in A_{i+}} \frac{\pd_{\theta}^{j_1}(a)^{q^{\mu_1}} \cdots \pd_{\theta}^{j_s}(a)^{q^{\mu_s}}}{a}
  = \frac{D_i^s}{L_i} \prod_{r=1}^s \kappa_{ij_r i} \bigl( \theta - \theta^{q^{\mu_r}}, \dots, \theta^{q^i} - \theta^{q^{\mu_r}} \bigr).
\end{equation}
Now for $0 \leqslant j \leqslant i$ and $\mu \geqslant 0$, we see from Remark~\ref{R:kappasubbed} that
\[
  \kappa_{iji} \bigl( \theta - \theta^{q^\mu}, \dots, \theta^{q^i} - \theta^{q^{\mu}} \bigr)
  = \frac{(-1)^{i-j} e_{i,i-j} \bigl( \theta - \theta^{q^{\mu}}, \dots, \theta^{q^{i-1}} - \theta^{q^{\mu}} \bigr)}
  {D_i},
\]
and the result follows by substituting into~\eqref{E:hypersuminter}.
\end{proof}

\begin{remark} \label{R:othertechniques}
As pointed out by the referee, Theorem~\ref{T:hypersum} can be obtained from Proposition~\ref{P:AnglesPellarin} by appropriately taking hyperderivatives with respect to the variables $t_1, \dots, t_s$ and specializing.  Furthermore, Pellarin and Perkins~\cite{PellarinPerkins20} have more general formulas for the sums in Proposition~\ref{P:AnglesPellarin}, which can be turned into formulas for the hyperderivative sums in Theorem~\ref{T:hypersum} that are valid for all $s \geqslant 1$.  We explore these ideas in~\S\ref{S:Anderson}.
\end{remark}

\begin{example}
We will see applications of Theorem~\ref{T:hypersum} to log-algebraicity calculations in \S\ref{S:Thakur}, but for the moment there are some interesting cases to observe.  For $s \leqslant q-1$, if we take $j_1 = \dots= j_s = j$, with $0 \leqslant j \leqslant i$, and $\mu_1 = \cdots = \mu_s = 0$, then we have
\begin{align*}
  \sum_{a \in A_{i+}} \frac{\pd_{\theta}^j(a)^s}{a} &= \frac{1}{L_i} \cdot (-1)^{s(i-j)} e_{i,i-j} \bigl( 0,\theta^q-\theta,\dots,
  \theta^{q^{i-1}}-\theta \bigr)^s \\
  &= \frac{(-1)^{s(i-j)}}{L_i} \cdot e_{i-1,i-j}([1],[2],\dots, [i-1])^s.
\end{align*}
\end{example}

\begin{example} \label{Ex:two}
Similarly, if we take $s=1$, $0 \leqslant j \leqslant i$, and $\mu\geqslant 0$, then
\[
  \sum_{a \in A_{i+}} \frac{\pd_{\theta}^j(a)^{q^{\mu}}}{a} = \frac{(-1)^{i-j}}{L_i} \cdot  e_{i,i-j} \bigl(
  \theta - \theta^{q^{\mu}}, \theta^q - \theta^{q^{\mu}}, \dots, \theta^{q^{i-1}} - \theta^{q^{\mu}} \bigr).
\]
\begin{enumerate}
\item[(a)] If $\mu \geqslant i$, then we have
\[
  \sum_{a \in A_{i+}} \frac{\pd_{\theta}^j(a)^{q^{\mu}}}{a} = \frac{(-1)^{i-j}}{L_i} \cdot e_{i,i-j} \bigl(
  -[\mu], -[\mu-1]^q, \dots, -[\mu-i+1]^{q^{i-1}} \bigr).
\]
\item[(b)] If $\mu < i$, then the simplification, using~\eqref{E:eijspec0}, is slightly different with
\begin{align*}
 \sum_{a \in A_{i+}} &\frac{\pd_{\theta}^j(a)^{q^{\mu}}}{a} \\
  &= \frac{(-1)^{i-j}}{L_i} \cdot e_{i,i-j} \bigl( -[\mu], -[\mu-1]^q, \dots, -[1]^{q^{\mu-1}},0,[1]^{q^{\mu}}, [2]^{q^{\mu}}, \dots, [i-\mu-1]^{q^{\mu}} \bigr) \notag \\
  &= \frac{(-1)^{i-j}}{L_i} \cdot e_{i-1,i-j} \bigl( -[\mu], -[\mu-1]^q, \dots, -[1]^{q^{\mu-1}},[1]^{q^{\mu}}, [2]^{q^{\mu}}, \dots, [i-\mu-1]^{q^{\mu}} \bigr). \notag
\end{align*}
\end{enumerate}
These cases will arise in the next section on log-algebraicity formulas.
\end{example}

\begin{remark}
More general hyperderivative power sum formulas can be used to analyze log-algebraicity formulas on tensor powers of the Carlitz module.  For more details on such formulas, see~\cite[Ch.~5, 9]{PLogAlg}.
\end{remark}

\section{Thakur's method for log-algebraicity on the Carlitz module} \label{S:Thakur}

In~\cite{And96}, Anderson developed the notion of log-algebraic power series identities for rank~$1$ Drinfeld modules based on previous special cases of Thakur~\cite{Thakur92}.  For the Carlitz module, Anderson's main theorem was the following.  We take $x$, $z$, and $\theta$ to be independent variables over $\FF_q$.

\begin{theorem}[{Anderson~\cite[Thm.~3, Prop.~8]{And96}}] \label{T:Anderson}
For $\beta(x) \in A[x]$, let
\begin{equation}
  \cP(\beta,z) \assign \exp_C \left( \sum_{a \in A_+} \frac{\beta(C_a(x))}{a} z^{q^{\deg a}} \right) \in \power{K[x]}{z}.
\end{equation}
Then in fact $\cP(\beta,z) \in A[x,z]$.
\end{theorem}

By using the Carlitz $A$-module operation, it is evident that $\cP(\beta,z)$ is $A$-linear in $\beta$, and so the values of $\cP(\beta,z)$ are completely determined by Anderson's \emph{special polynomials},
\[
  P_m(x,z) \assign \cP(x^m,z), \quad m \geqslant 0.
\]
Anderson~\cite[Prop.~8]{And96} derives several properties of special polynomials, including bounds for their degrees in $x$, $z$, and $\theta$.  However, Anderson's proof was indirect, and so except in certain cases he does not provide exact formulas for $P_m(x,z)$.  In~\cite[\S 8.10]{Thakur}, Thakur constructed a new method for proving Anderson's theorem using the power sum formulas for $S_i(k)$ (Theorem~\ref{T:CarlitzLee} and others) and Carlitz's formulas for $\brac{a}{k}$ in~\eqref{E:EbracCarlitz}.  One benefit of Thakur's method was that he derived exact formulas for $P_m(x,z)$ for a large class of values of $m$.  In this section we restate Thakur's results using symmetric polynomials and provide a proof using the techniques of sections~\S\ref{S:Brac}--\ref{S:Sums}.

\begin{theorem}[{Thakur~\cite[\S 8.10]{Thakur}}] \label{T:Thakur} \
\begin{enumerate}
\item[(a)] Let $m = q^{\mu}$ for $\mu \geqslant 0$.  Then
\[
  P_{q^\mu}(x,z) = \sum_{d=0}^{\mu} (-1)^{\mu-d} \cdot C_{e_{\mu,\mu-d}(\theta, \dots, \theta^{q^{\mu-1}})} \bigl(
  C_{\theta^d}(x)\cdot z \bigr).
\]
\item[(b)] Suppose $m$ is of the form $m = q^{\mu_1} + \cdots + q^{\mu_s}$ for $\mu_j \geqslant 0$ and $1 \leqslant s \leqslant q-1$.  Then
\[
  P_m(x,z) = \sum_{d_1=0}^{\mu_1} \cdots \sum_{d_s=0}^{\mu_s} (-1)^{\sum \mu_j - \sum d_j} \cdot
  C_{\prod_{r=1}^s e_{\mu_r,\mu_r-d_r} (\theta, \dots, \theta^{q^{\mu_r-1}})}
  \Biggl( \Biggl( \prod_{r=1}^s C_{\theta^{d_r}}(x) \Biggr)\cdot z \Biggr).
\]
\end{enumerate}
\end{theorem}

\begin{remark}
(a) Also in~\cite[\S 8.10]{Thakur}, Thakur gives a proof that $P_m(x,z) \in K[x,z]$ for general $m$ using the same methods, though the intricate calculations present difficulties to conclude that the coefficients are in $A$.  Following suggestions of the referee, we prove Theorem~\ref{T:Anderson} and thus that $P_m(x,z) \in A[x,z]$ for all~$m$, though at the expense of the explicitness of Theorem~\ref{T:Thakur}.  (b) Multivariable versions of Theorem~\ref{T:Anderson} and parts of Theorem~\ref{T:Thakur} were worked out by Angl\`{e}s, Pellarin, and Tavares Ribeiro~\cite{AnglesPellarinTavares18} based on earlier work of Pellarin~\cite{Pellarin12}, and it would be interesting to investigate their constructions using the techniques in this section.  (c)~Further results on connections between polynomial power sums and log-algebraicity results can be found in~\cite{AnglesTavares17}, \cite{CEP18}, \cite{DemeslayPhD}, \cite{GreenP18}, \cite{Perkins14a}, \cite{Perkins14c}.  (d)~While Thakur's methods do not readily transfer to the setting of tensor powers of the Carlitz module, generalizations of our proof here of Theorem~\ref{T:Thakur} to higher tensor powers will be the subject of~\cite[Ch.~9]{PLogAlg}.
\end{remark}

\begin{example} \label{Ex:Thakurexamples}
As observed by Thakur~\cite[Rmk.~8.10.1]{Thakur}, we have $P_1(x,z) = xz$, $P_q(x,z) = x^q z - x^q z^q$, and (for $q > 2$) $P_{q+1}(x,z) = x^{q+1}z - x^{2q} z^q$.  If we consider $m = 2q$ for $q >2$, then Theorem~\ref{T:Thakur}(b) implies
\begin{align*}
  P_{2q}(x,z) &= C_{e_{11}(\theta) e_{11}(\theta)}(x^2z)
  - 2 C_{e_{11}(\theta)}\bigl( C_{\theta}(x)\cdot xz \bigr) + C_{\theta}(x)^2 \cdot z \\
  &= C_{\theta^2}(x^2 z) - 2C_{\theta}\bigl( (\theta x^2 + x^{q+1}) z\bigr) + (\theta x + x^q)^2 z \\
  &= x^{2q} z - \bigl( (\theta^q -\theta) x^{2q} + 2x^{q^2+q} \bigr)z^q + x^{2q^2} z^{q^2},
\end{align*}
which agrees with~\cite[Eq.~(27)]{And96}.  For other examples using Theorem~\ref{T:Thakur}, we find for $q>2$,
\begin{align*}
  P_{q^2}(x,z) &= C_{\theta^2}(x) \cdot z - C_{\theta^q + \theta}\bigl( C_{\theta}(x) \cdot z \bigr) + C_{\theta^{q+1}}(xz), \\
  P_{q^2+1}(x,z) &= C_{\theta^2}(x)\cdot xz - C_{\theta^q + \theta}\bigl( C_{\theta}(x)\cdot xz \bigr) + C_{\theta^{q+1}}(x^2 z), \\
  P_{q^2+q}(x,z) &= \begin{aligned}[t]
  C_{\theta}&(x)\cdot C_{\theta^2}(x)\cdot z - C_{\theta}\bigl( C_{\theta^2}(x)\cdot xz \bigr) - C_{\theta^q+\theta} \bigl( C_{\theta}(x)^2 \cdot z \bigr) \\& {}+ C_{\theta^{q+1}+\theta^2} \bigl( C_{\theta}(x)\cdot xz \bigr)+ C_{\theta^{q+1}} \bigl( C_{\theta}(x)\cdot xz \bigr) - C_{\theta^{q+2}}(x^2z),
  \end{aligned} \\
  P_{2q^2}(x,z) &= \begin{aligned}[t]
  C_{\theta^2}&(x)^2 \cdot z - C_{\theta^q+\theta} \bigl( C_{\theta_2}(x) \cdot C_{\theta}(x) \cdot z \bigr)
  + C_{\theta^{2q} + 2\theta^{q+1} + \theta^2}\bigl( C_{\theta}(x)^2\cdot z \bigr) \\
  &{}+C_{\theta^{q+1}} \bigl( C_{\theta^2}(x)\cdot z \bigr) - C_{\theta^{2q+1} + \theta^{q+2}} \bigl( C_{\theta}(x)\cdot xz \bigr)
  +C_{\theta^{2q+2}}(xz),
  \end{aligned}
  \\
  P_{q^3}(x,z) &= \begin{aligned}[t]
  C_{\theta^3}&(x)\cdot z - C_{\theta^{q^2} + \theta^q + \theta} \bigl( C_{\theta^2}(x)\cdot z \bigr)
    + C_{\theta^{q^2+q} + \theta^{q^2+1} + \theta^{q+1}} \bigl(C_{\theta}(x)\cdot z \bigr) \\
    & {}- C_{\theta^{q^2+q+1}}(xz).
  \end{aligned}
\end{align*}
\end{example}

The remainder of this section is devoted to a new proof of Theorem~\ref{T:Thakur}.  We first need some lemmas on symmetric polynomials.

\begin{lemma} \label{L:ehdiff}
Let $d \geqslant 1$, and let $T$ be a variable independent from $x_1, \dots, x_d$.
\begin{enumerate}
\item[(a)] For $0 \leqslant k \leqslant d$,
\[
  e_{d,d-k}(T-x_1, \dots, T-x_d) = \sum_{j=k}^{d} (-1)^{d-j} \binom{j}{k} e_{d,d-j}(x_1,\dots, x_d) T^{j-k}.
\]
\item[(b)] For $k \geqslant 0$,
\[
  h_{d,k}(T-x_1, \dots, T-x_d) = \sum_{j=0}^k (-1)^j \binom{d+k-1}{k-j} h_{d,j}(x_1, \dots, x_d) T^{k-j}.
\]
\end{enumerate}
\end{lemma}

\begin{proof}
Both identities are proved by taking hyperderivatives with respect to~$T$.  For (a), consider the polynomial
\[
P(T) = (T-x_1) \cdots (T-x_d) = \sum_{j=0}^d (-1)^{d-j} e_{d,d-j}(x_1, \dots, x_d) T^j.
\]
By the product rule (see \cite[\S 2.2]{Jeong11} or \cite[\S 2.3]{PLogAlg})
\begin{align*}
  \pd_{T}^k (P(x)) &= \sum_{\substack{k_1, \dots, k_d \geqslant 0 \\ k_1 + \dots + k_d = k }} \pd_{T}^{k_1} (T-x_1) \cdots \pd_T^{k_d}(T-x_d) \\
  &= e_{d,d-k}(T-x_1, \dots, T-x_d).
\end{align*}
On the other hand,
\[
  \pd_{T}^k (P(x)) = \sum_{j=k}^d (-1)^{d-j} \binom{j}{k} e_{d,d-j}(x_1, \dots, x_d) T^{j-k}.
\]

For (b) we proceed by a double induction.  We first note that the result holds for all $k \geqslant 0$ for $h_{1,k} = (T-x_1)^k$ by the binomial theorem.  Now suppose for some $d_0 \geqslant 2$ that the result holds for all $(d,k)$ in the set $\{ (d,k) \mid d \leqslant d_0 - 1,\ k \geqslant 0 \}$.  Now $h_{d_0,0}(x_1, \dots, x_d) = 1$, so the result also holds for $(d,k) = (d_0,0)$, so suppose further that there is some $k_0 \geqslant 1$ so that the result holds for all $(d_0,k)$ with $k \leqslant k_0 - 1$.  If we let $\ox_{i}$ denote the tuple $(x_1, \dots, x_i)$, then using~\eqref{E:hijrec1} the induction hypothesis implies
\begin{align*}
  h_{d_0,k_0}(T-x_1, \dots, T-x_d) &=
  \begin{aligned}[t]
  h_{d_0-1,k_0}(&T-x_1, \dots, T-x_{d_0-1}) \\
  &{} + (T-x_{d_0}) h_{d_0,k_0-1} (T-x_1, \dots, T-x_{d_0})
  \end{aligned}
  \\
  &= \begin{aligned}[t]
  &\sum_{j=0}^{k_0} (-1)^{j} \binom{d_0+k_0-2}{k_0-j} h_{d_0-1,j}(\ox_{d_0-1}) T^{k_0-j} \\
  &{}+ (T-x_{d_0}) \sum_{j = 1}^{k_0} (-1)^{j-1} \binom{d_0 + k_0-2}{k_0-j} h_{d_0,j-1}(\ox_{d_0}) T^{k_0-j}.
  \end{aligned}
\end{align*}
Collecting terms and using~\eqref{E:hijrec1} again, we find that
\begin{multline*}
  h_{d_0,k_0}(T-x_1, \dots, T-x_d) = \sum_{j=0}^{k_0} (-1)^j \biggl( \binom{d_0+k_0-2}{k_0-j} +
   \binom{d_0+k_0-2}{k_0-j-1} \biggr) \\
    {} \times h_{d_0,j}(\ox_{d_0}) T^{k_0-j},
\end{multline*}
and the result follows in the case $(d,k) = (d_0,k_0)$.
\end{proof}

Now for $i \geqslant 1$ and $1 \leqslant \ell \leqslant k \leqslant i-1$, define
\begin{equation} \label{E:Gikldef}
  G_{i,k,\ell} = \sum_{j=k}^i (-1)^{i-j} e_{i-1,i-j}(y_1, \dots, y_\ell, x_{\ell+1}, \dots, x_{i-1}) \cdot
  h_{k,j-k}(x_1, \dots, x_k),
\end{equation}
where $x_1, \dots, x_{i-1}, y_1, \dots, y_{\ell}$ are independent variables over $\ZZ$.  To emphasize the order of the variables when making substitutions, we will write
\[
G_{i,k,\ell} = G_{i,k,\ell}(x_1, \dots, x_k; x_{k+1}, \dots, x_{i-1}; y_1, \dots, y_{\ell}).
\]
Ostensibly these polynomials are in $\ZZ[x_1, \dots, x_{i-1}, y_1, \dots, y_{\ell}]$, but in fact they do not contain the variables $x_{\ell+1}, \dots, x_{k}$, as shown in the following lemma.

\begin{lemma} \label{L:Gikl}
For $i \geqslant 1$ and $1 \leqslant \ell \leqslant k \leqslant i-1$,
\begin{multline*}
  G_{i,k,\ell}(x_1, \dots, x_k; x_{k+1}, \dots, x_{i-1}; y_1, \dots, y_{\ell}) \\
  = G_{i-(k-\ell),\ell,\ell}(x_1, \dots, x_{\ell}; x_{k+1}, \dots, x_{i-1}; y_1, \dots, y_{\ell}).
\end{multline*}
\end{lemma}

\begin{proof}
A short calculation reveals that the right-hand side of this formula is obtained from the left by substituting in $x_{\ell+1} = \cdots = x_{k} = 0$, so the proof reduces to showing that $G_{i,k,\ell}$ does not actually contain $x_{\ell+1}, \dots, x_k$.  As the defining expression for $G_{i,k,\ell}$ is symmetric in $x_{\ell+1}, \dots, x_k$, it suffices to show that it does not contain $x_k$ when $\ell < k$.  (When $\ell=k$, there is nothing to prove.)  Define the following tuples:
\[
  \oy_{\ell} = (y_1, \dots, y_{\ell}), \quad \ox_{\ell,k} = (x_{\ell+1}, \dots, \widehat{x_k}, \dots, x_{i-1}),
  \quad \ox_{j} = (x_1, \dots, x_{j}).
\]
Then, since $h_{kj}(x_1, \dots, x_k) = \sum_{n=0}^j h_{k-1,j-n}(\ox_{k-1}) x_k^n$ (by~\eqref{E:hijrec1}),
\begin{align*}
  \frac{\pd}{\pd x_k} ( G_{i,k,\ell}) &=
  \begin{aligned}[t]
  \sum_{j=k}^i (-1)^{i-j} &\biggl( e_{i-2,i-j-1}(\oy_{\ell},\ox_{\ell,k}) \cdot
    h_{k,j-k}(\ox_k) \\
    &{}+ e_{i-1,i-j}(\oy_{\ell}, x_{\ell+1}, \dots, x_{i-1}) \cdot \sum_{n=0}^{j-k} n\cdot h_{k-1,j-k-n}(\ox_{k-1}) x_{k}^{n-1} \biggr)
  \end{aligned}
  \\
  &= \begin{aligned}[t]
  \sum_{j=k}^i (-1)^{i-j} & \biggl( e_{i-2,i-j-1}(\oy_{\ell},\ox_{\ell,k}) \sum_{n=0}^{j-k} h_{k-1,j-k-n}(\ox_{k-1}) x_k^{n} \\
  &{}+ \bigl( e_{i-2,i-j}(\oy_{\ell},\ox_{\ell,k}) + x_k e_{i-2,i-j-1}(\oy_{\ell},\ox_{\ell,k})) \\
  &{}\times \sum_{n=0}^{j-k} n \cdot h_{k-1,j-k-n}(\ox_{k-1}) x_k^{n-1} \biggr),
  \end{aligned}
\end{align*}
after applying \eqref{E:eijrec1}.  By rearranging terms and reordering the sums we finally obtain
\begin{multline*}
  \frac{\pd}{\pd x_k} ( G_{i,k,\ell}) = \sum_{n=0}^{i-k} \Biggl( \sum_{j=n+k}^i (-1)^{i-j} \Bigl( e_{i-2,i-j-1}(\oy_{\ell}, \ox_{\ell,k}) h_{k-1,j-k-n} (\ox_{k-1}) \\
  {}+ e_{i-2,i-j} (\oy_{\ell}, \ox_{\ell,k}) h_{k-1,j-k-n-1}(\ox_{k-1}) \Bigr) \Biggr) (n+1) x_k^n.
\end{multline*}
The inner sum telescopes leaving only
\begin{multline*}
  \frac{\pd}{\pd x_k} ( G_{i,k,\ell}) = (-1)^{i-n-k} e_{i-2,i-n-k}(\oy_{\ell},\ox_{\ell,k}) \cdot h_{k-1,-1}(\ox_{k-1}) \\
  {}+ e_{i-2,-1}(\oy_{\ell},\ox_{\ell,k}) \cdot h_{k-1,i-k-n}(\ox_{k-1}),
\end{multline*}
and since both of these terms are zero, as they have terms with negative indices, we see that $(\pd/\pd x_k)(G_{i,k,\ell}) \equiv 0$ identically.  Thus $G_{i,k,\ell}$ does not contain the variable~$x_k$.
\end{proof}

\begin{proof}[Proof of Theorem~\ref{T:Thakur}(a)]
Although part (a) of Theorem~\ref{T:Thakur} is a special case of part (b), once complete the argument for (a) will simplify the argument for (b).  For $i \geqslant 0$, we let
\begin{equation} \label{E:lambdadef}
  \lambda_i(q^{\mu}) = \sum_{a \in A_{i+}} \frac{C_a(x)^{q^\mu}}{a},
\end{equation}
so that
\begin{equation} \label{E:lambdasum}
P_{q^\mu}(x,z) = \exp_C \biggl( \sum_{i=0}^{\infty} \lambda_i (q^{\mu}) z^{q^i} \biggr).
\end{equation}
The major focus of the proof is to find a simplification for $\lambda_i(q^{\mu})$.  Using Proposition~\ref{P:newbrac}, we see that
\begin{align*}
  \lambda_i(q^{\mu}) &= \sum_{a \in A_{i+}} \frac{1}{a} \sum_{k=0}^i \brac{a}{k}^{q^{\mu}} x^{q^{k+\mu}} \\
  &= \sum_{a \in A_{i+}} \frac{1}{a} \sum_{k=0}^{i} \sum_{j=k}^{i} \pd^j_{\theta}(a)^{q^{\mu}} \cdot
  h_{k,j-k}([1], \dots, [k])^{q^{\mu}} \cdot x^{q^{k+\mu}}.
\end{align*}
By reordering the sum and applying the calculation from Example~\ref{Ex:two}, we see that
\begin{align} \label{E:newbracsub}
  \lambda_i(q^{\mu}) &= \sum_{k=0}^i \sum_{j=k}^i h_{k,j-k}\bigl( [1]^{q^\mu}, \dots, [k]^{q^\mu} \bigr) \cdot
  x^{q^{k+\mu}} \sum_{a \in A_{i+}} \frac{ \pd_{\theta}^j(a)^{q^{\mu}}}{a} \\
  &= \begin{aligned}[t]
  \sum_{k=0}^i & \sum_{j=k}^i h_{k,j-k}\bigl( [1]^{q^\mu}, \dots, [k]^{q^\mu} \bigr) \cdot
  x^{q^{k+\mu}} \\
  & {}\times \frac{(-1)^{i-j}}{L_i} \cdot e_{i,i-j} \bigl( \theta - \theta^{q^{\mu}}, \theta^q - \theta^{q^{\mu}},
  \dots, \theta^{q^{i-1}} - \theta^{q^{\mu}} \bigr).
  \end{aligned} \notag
\end{align}
Just as in Example~\ref{Ex:two}, there are the two cases $i \leqslant \mu$ and $i > \mu$.  We will consider here the case $i > \mu$.  The case where $i \leqslant \mu$ is similar with the same resulting formula, and for the sake of space we leave it to the reader.  Using the formula in Example~\ref{Ex:two}(b), we then see that
\begin{multline} \label{E:simplifysub}
  \lambda_i(q^{\mu}) = \sum_{k=0}^i \sum_{j=k}^i (-1)^{i-j} e_{i-1,i-j} \bigl( -[\mu], \dots, -[1]^{q^{\mu-1}},
  [1]^{q^{\mu}}, \dots, [i-\mu-1]^{q^{\mu}} \bigr) \\
  {}\times h_{k,j-k} \bigl( [1]^{q^{\mu}}, \dots, [k]^{q^{\mu}} \bigr) \cdot \frac{x^{q^{k+\mu}}}{L_i}.
\end{multline}
By Proposition~\ref{P:symmrec2}, the inner sum vanishes when $k \leqslant i - \mu -1$, and so using~\eqref{E:Gikldef} we have
\begin{multline*}
  \lambda_i(q^{\mu}) = \sum_{k=i-\mu}^i G_{i,k,k-i+\mu+1} \bigl( [k]^{q^{\mu}}, \dots, [1]^{q^{\mu}};
  -[1]^{q^{\mu-1}}, \dots, -[i-k-1]^{q^{\mu+k-i+1}}; \\ -[i-k]^{q^{\mu+k-i}}, \dots, -[\mu] \bigr) \cdot \frac{x^{q^{k+\mu}}}{L_i}.
\end{multline*}
Lemma~\ref{L:Gikl} then implies that if we set $\ell_k \assign k-i+\mu+1$, then
\begin{align*}
  \lambda_i(q^{\mu}) &=
  \begin{aligned}[t]
  \sum_{k=i-\mu}^i G_{\mu+1,\ell_k,\ell_k} \bigl( [k]^{q^{\mu}},  \dots, [i-\mu]^{q^{\mu}};
  -[1]^{q^{\mu-1}}, & \dots, -[i-k-1]^{q^{\mu+k-i+1}}; \\
  &-[i-k]^{q^{\mu+k-i}}, \dots, -[\mu] \bigr) \cdot \frac{x^{q^{k+\mu}}}{L_i}.
  \end{aligned}
  \\
  &= \begin{aligned}[t]
  \sum_{\ell=1}^{\mu+1} G_{\mu+1,\ell,\ell} \bigl( [i-\mu+\ell-1]^{q^{\mu}}, \dots, & [i - \mu]^{q^{\mu}};
  -[1]^{q^{\mu-1}}, \dots, -[\mu-\ell]^{q^{\ell}}; \\
  & -[\mu-\ell+1]^{q^{\ell-1}}, \dots, -[\mu] \bigr) \cdot \frac{x^{q^{i+\ell-1}}}{L_i}.
  \end{aligned}
\end{align*}
Although this looks complicated, the important thing is that the number of variables in $G_{\mu+1,\ell,\ell}$ is bounded independent of~$i$.  Moreover, applying the definition of $G_{\mu+1,\ell,\ell}$ from~\eqref{E:Gikldef} and reordering the sum, we find
\begin{multline} \label{E:lambdainter}
  \lambda_i(q^{\mu}) =
    \sum_{j=1}^{\mu+1} \sum_{\ell=1}^{j} (-1)^{\mu+1-j} e_{\mu,\mu+1-j} \bigl( -[1]^{q^{\mu-1}}, -[2]^{q^{\mu-2}}, \dots, -[\mu] \bigr) \\
     {} \times h_{\ell,j-\ell} \bigl( [i-\mu+\ell-1]^{q^{\mu}}, \dots, [i-\mu]^{q^{\mu}} \bigr) \cdot
    \frac{x^{q^{i+\ell-1}}}{L_i}.
\end{multline}
Now by Proposition~\ref{P:bracthetam} and Lemma~\ref{L:ehdiff}(b), we see that
\begin{align*}
  h_{\ell,j-\ell} &\bigl( [i-\mu+\ell-1]^{q^{\mu}}, \dots, [i-\mu]^{q^{\mu}} \bigr) \\
  &= h_{\ell,j-\ell} \bigl( \theta^{q^{i+\ell-1}} - \theta^{q^{\mu}}, \dots, \theta^{q^{i}} - \theta^{q^{\mu}} \bigr) \\
  &= (-1)^{j-\ell} \sum_{n=0}^{j-\ell} (-1)^n \binom{j-1}{j-\ell-n} h_{\ell,n} \bigl(\theta, \theta^q, \dots, \theta^{q^{\ell-1}} \bigr)^{q^i} \cdot \theta^{(j-\ell-n)q^{\mu}} \\
  &= \sum_{n=0}^{j-\ell} (-1)^{j-\ell+n} \binom{j-1}{j-\ell-n} \theta^{(j-\ell-n)q^{\mu}} \cdot
  \brac{\theta^{\ell+n-1}}{\ell-1}^{q^i}.
\end{align*}
Therefore, after some reordering and reindexing of sums,
\begin{align} \label{E:hsum}
  \sum_{\ell=1}^j h_{\ell,j-\ell} &\bigl( [i-\mu+\ell-1]^{q^{\mu}}, \dots, [i-\mu]^{q^{\mu}} \bigr) \cdot
  \frac{x^{q^{i+\ell-1}}}{L_i} \\
  &= \sum_{n=0}^{j-1} \sum_{\ell=1}^{j-n} (-1)^{j-\ell+n} \binom{j-1}{j-\ell-n} \theta^{(j-\ell-n)q^{\mu}} \cdot \brac{\theta^{\ell+n-1}}{\ell-1}^{q^i} \cdot
  \frac{x^{q^{i+\ell-1}}}{L_i} \notag \\
  &= \sum_{n=0}^{j-1} \sum_{d=n}^{j-1} (-1)^{j-d-1} \binom{j-1}{j-d-1} \theta^{(j-d-1)q^{\mu}} \cdot \brac{\theta^{d}}{d-n}^{q^i} \cdot
  \frac{x^{q^{i+d-n}}}{L_i} \notag \\
  &= \sum_{d=0}^{j-1} (-1)^{j-d-1} \binom{j-1}{d} \theta^{(j-d-1)q^{\mu}} \sum_{n=0}^d
  \brac{\theta^d}{d-n}^{q^i} \cdot \frac{ \bigl(x^{q^{d-n}}\bigr)^{q^i}}{L_i}. \notag \\
  &= \sum_{d=0}^{j-1} (-1)^{j-d-1} \binom{j-1}{d} \theta^{(j-d-1)q^{\mu}} \cdot
  \frac{C_{\theta^d}(x)^{q^i}}{L_i}. \notag
\end{align}
We then substitute \eqref{E:hsum} into~\eqref{E:lambdainter}, and find
\begin{align} \label{E:lambdafinal}
  \lambda_i(q^{\mu}) &=
  \begin{aligned}[t]
     \sum_{j=1}^{\mu+1} &\ e_{\mu,\mu+1-j}([1]^{q^{\mu-1}}, \dots, [\mu]) \\
     & {} \times \sum_{d=0}^{j-1} (-1)^{j-d-1} \binom{j-1}{d} \theta^{(j-d-1)q^{\mu}} \cdot
     \frac{C_{\theta^d}(x)^{q^i}}{L_i}
  \end{aligned}
  \\
  &= \sum_{j=0}^{\mu} \sum_{d=0}^{j} (-1)^{j-d} e_{\mu,\mu-j}([1]^{q^{\mu-1}},\dots, [\mu])
     \binom{j}{d} \theta^{(j-d)q^{\mu}} \cdot \frac{C_{\theta^d}(x)^{q^i}}{L_i} \notag \\
  &= \sum_{d=0}^{\mu} (-1)^{\mu-d} e_{\mu,\mu-d} \bigl( \theta^{q^{\mu}} - [1]^{q^{\mu-1}}, \theta^{q^{\mu}}
  -[2]^{q^{\mu-2}}, \dots, \theta^{q^{\mu}} - [\mu] \bigr) \cdot \frac{C_{\theta^d}(x)^{q^i}}{L_i}  \notag \\
  &= \sum_{d=0}^{\mu} (-1)^{\mu-d} e_{\mu,\mu-d} \bigl( \theta^{q^{\mu-1}}, \theta^{q^{\mu-2}}, \dots, \theta) \cdot \frac{C_{\theta^d}(x)^{q^i}}{L_i}, \notag
\end{align}
where in the third equality we have applied Lemma~\ref{L:ehdiff}(a).  From this we see that
\begin{equation}
  \sum_{i=0}^{\infty} \lambda_i(q^{\mu}) z^{q^i} =
  \sum_{d=0}^{\mu} (-1)^{\mu-d} e_{\mu,\mu-d} \bigl(\theta, \theta^q, \dots, \theta^{q^{\mu-1}}) \cdot \log_C \bigl(
  C_{\theta^d}(x) z \bigr),
\end{equation}
which after applying $\exp_C(z)$ to both sides yields the desired result.
\end{proof}

\begin{proof}[Proof of theorem~\ref{T:Thakur}(b)]
The proof of part (b) runs along the same lines as part (a), but with more bookkeeping.  Nevertheless, we have carried out much of the difficult calculations in (a).  As in~\eqref{E:lambdadef}, we set
\[
  \lambda_i(m) = \sum_{a \in A_{i+}} \frac{C_a(x)^m}{a},
\]
so that $P_m(x,z) = \exp_C \bigl( \sum_{i=0}^{\infty} \lambda_i(m) z^{q^i} \bigr)$.  Now since $m = q^{\mu_1} + \dots + q^{\mu_s}$, we see that
\begin{align} \label{E:lambdaim}
  \lambda_i(m) &= \sum_{a \in A_{i+}} \Biggl( \sum_{k_1=0}^i \brac{a}{k_1}^{q^{\mu_1}} x^{q^{k_1+\mu_1}} \Biggr) \cdots \Biggl( \sum_{k_s=0}^i \brac{a}{k_s}^{q^{\mu_s}} x^{q^{k_s+\mu_s}} \Biggr) \frac{1}{a} \\
  &= \sum_{k_1=0}^i \cdots \sum_{k_s=0}^i \Biggl( \sum_{a \in A_{i+}} \brac{a}{k_1}^{q^{\mu_1}} \cdots
    \brac{a}{k_s}^{q^{\mu_s}} \cdot \frac{1}{a} \Biggr) x^{q^{k_1+\mu_1} + \cdots + q^{k_s+\mu_s}}. \notag
\end{align}
We apply Proposition~\ref{P:newbrac} to the inner sum and find as in~\eqref{E:newbracsub},
\begin{align} \label{E:bracsum}
   \sum_{a \in A_{i+}} &\ \brac{a}{k_1}^{q^{\mu_1}} \cdots \brac{a}{k_s}^{q^{\mu_s}} \cdot \frac{1}{a} \\
   &= \begin{aligned}[t]
     \sum_{j_1=k_1}^i \cdots \sum_{j_s=k_s}^i &\ h_{k_1,j_1-k_1} \bigl( [1], \dots, [k_1] \bigr)^{q^{\mu_1}} \cdots h_{k_s,j_s-k_s} \bigl([1], \dots, [k_s] \bigr)^{q^{\mu_s}} \\
     &{} \times \sum_{a \in A_{i+}} \frac{\pd_{\theta}^{j_1}(a)^{q^{\mu_1}} \cdots \pd_{\theta}^{j_s}(a)^{q^{\mu_s}}}{a}. \notag
   \end{aligned}
\end{align}
Now the value of the final hyperderivative sum has been obtained in Theorem~\ref{T:hypersum}, and we can use the methods of Example~\ref{Ex:two} to simplify it as we did in~\eqref{E:newbracsub} and~\eqref{E:simplifysub}.  Again the cases where $i \leqslant \mu_r$ are similar with the same resulting formula, but for the purpose of space we leave the details to the reader.  We assume then that $i > \mu_r$ for each $r=1, \dots, s$, and obtain that
\begin{align*}
  \lambda_i(m) &= \begin{aligned}[t]
    \sum_{k_1=0}^i &{}\cdots \sum_{k_s=0}^i \frac{x^{q^{k_1+\mu_1} + \cdots + q^{k_s+\mu_s}}}{L_i}
    \Biggl( \sum_{j_1=0}^{k_1} \cdots \sum_{j_s=0}^{k_s} \cdot \prod_{r=1}^{s} (-1)^{i-j_r} \\
    &{}\times e_{i-1,i-j_r} \bigl(-[\mu_r], \dots, -[1]^{q^{\mu_r-1}}, [1]^{q^{\mu_r}}, \dots, [i-\mu_r-1]^{q^{\mu_r}} \bigr) \\
    &{}\times h_{k_r,j_r-k_r} \bigl( [1]^{q^{\mu_r}}, \dots, [k_r]^{q^{\mu_r}} \bigr) \Biggr)
  \end{aligned}
  \\
  &= \begin{aligned}[t]
    \frac{1}{L_i} \prod_{r=1}^{s} &\Biggl( \sum_{k_r=0}^{i} \sum_{j_r=k_r}^{i} (-1)^{i-j_r} e_{i-1,i-j_r} \bigl(
    -[\mu_r], \dots, [i-\mu_r-1]^{q^{\mu_r}}\bigr) \\
    & {}\times h_{k_r,j_r-k_r} \bigl( [1]^{q^{\mu_r}}, \dots, [k_r]^{q^{\mu_r}} \bigr) \cdot x^{q^{k_r+\mu_r}} \Biggr).
  \end{aligned}
\end{align*}
The inner double sum has already been evaluated in the proof of part (a), starting with~\eqref{E:simplifysub}.  Therefore, using~\eqref{E:lambdafinal}, we obtain
\begin{equation}
  \lambda_i(m) = \frac{1}{L_i} \prod_{r=1}^s \sum_{d_r=0}^{\mu_r} (-1)^{\mu_r-d_r} e_{\mu_r,\mu_r-d_r} \bigl( \theta^{q^{\mu_r-1}}, \theta^{q^{\mu_r-2}}, \dots, \theta \bigr) \cdot C_{\theta^{d_r}}(x)^{q^i},
\end{equation}
and so
\begin{multline}
  \sum_{i=0}^{\infty} \lambda_i(m) z^{q^i} = \sum_{d_1=0}^{\mu_1} \cdots \sum_{d_s=0}^{\mu_s} (-1)^{\sum \mu_r - \sum d_r} \\
  {}\times \prod_{r=1}^s e_{\mu_r,\mu_r-d_r} \bigl(\theta, \theta^q, \dots, \theta^{q^{\mu_r-1}} \bigr)
  \cdot \log_C \Biggl( \Biggl( \prod_{r=1}^{s} C_{\theta^{d_r}}(x) \Biggr) \cdot z \Biggr),
\end{multline}
which yields the desired result after exponentiation.
\end{proof}

\section{Proof of Anderson's theorem} \label{S:Anderson}

At the suggestion of the referee, in this section we demonstrate how the techniques in this paper in conjunction with results of Pellarin and Perkins~\cite{PellarinPerkins20} can be used to prove Anderson's Theorem~\ref{T:Anderson}.  In particular we show that $P_m(x,z) = \cP(x^m,z) \in A[x,z]$ for all $m \geqslant 0$.

For $s$, $i \geqslant 0$, we let
\begin{equation} \label{E:Ssidef}
  S_{s,i} \assign \sum_{a \in A_{i+}} \frac{a(t_1) \cdots a(t_s)}{a}
\end{equation}
and
\begin{equation} \label{E:Fsidef}
  F_{s,i+1} \assign \sum_{j = 0}^i S_{s,j}.
\end{equation}
The proof of Theorem~\ref{T:Thakur} hinged on the hyperderivative formulas in Theorem~\ref{T:hypersum}, which were restricted to $1 \leqslant s \leqslant q-1$.  However, as mentioned in Remark~\ref{R:othertechniques} it is possible to derive Theorem~\ref{T:hypersum} from Proposition~\ref{P:AnglesPellarin}.  More specifically, in the notation of the statement of Theorem~\ref{T:hypersum}, one finds from Lemma~\ref{L:Taylor} that
\begin{equation} \label{E:hypersumred1}
  \sum_{a \in A_{i+}} \frac{\pd_{\theta}^{j_1}(a)^{q^{\mu_1}} \cdots \pd_{\theta}^{j_s}(a)^{q^{\mu_s}}}{a}
  = \bigl( \pd_{t_1}^{j_1} \circ \cdots \circ \pd_{t_s}^{j_s}\bigr) \bigl( S_{s,i} \bigr) \big|_{t_1=\theta^{q^{\mu_1}}, \ldots, t_s=\theta^{q^{\mu_s}}}.
\end{equation}
However, Proposition~\ref{P:AnglesPellarin} with $1 \leqslant s \leqslant q-1$, combined with the product rule, yields
\begin{equation} \label{E:Ssired1}
  \bigl( \pd_{t_1}^{j_1} \circ \cdots \circ \pd_{t_s}^{j_s}\bigr) \bigl( S_{s,i} \bigr) = \frac{1}{L_i} \prod_{r=1}^{s} (-1)^{i-j_r} e_{i,i-j_r} \bigl( \theta-t_r, \theta^q-t_r, \ldots, \theta^{q^{i-1}}-t_r \bigr),
\end{equation}
from which Theorem~\ref{T:hypersum} follows after substitution.

Now the formula in~\eqref{E:hypersumred1} holds for any $s \geqslant 1$, and it is the subsequent formula in \eqref{E:Ssired1} for $1 \leqslant s \leqslant q-1$ that allows for the proof of Theorem~\ref{T:Thakur}.  There are formulas for $S_{s,i}$ that hold for arbitrary $s$ in \cite{PellarinPerkins20} that lose some of the precision of Proposition~\ref{P:AnglesPellarin} but that can still be used to derive formulas like \eqref{E:Ssired1} in these more general cases.  Because of this diminished precision the resulting identities for~$P_m(x,z)$ are not as explicit as in Thakur's Theorem~\ref{T:Thakur}, but we retain enough information to prove Anderson's Theorem~\ref{T:Anderson}.

Now as in Theorem~\ref{T:Thakur}, we can express $m = q^{\mu_1} + \dots + q^{\mu_s}$, but since we will allow arbitrary~$s$, it will be simpler to take $\mu_1 = \cdots = \mu_s = 0$.  Thus, $m=s$ throughout this section.  (This also bypasses the need for the $G_{i,k,\ell}$ polynomials of \S\ref{S:Thakur}.)  Because Theorem~\ref{T:Thakur} already covers the case $s \leqslant q-1$, we will also assume that $s \geqslant q$.

Before presenting the proof of Theorem~\ref{T:Anderson}, we focus on identities for $S_{s,i}$ and $F_{s,i+1}$ derived from~\cite{PellarinPerkins20}.  Indeed~\cite[Thm.~7]{PellarinPerkins20} provides a formula for $S_{s,i}$ when $i \geqslant \lfloor (s-1)/(q-1) \rfloor$ and arbitrary $s$, but in order to show that $P_s(x,z) \in A[x,z]$ we need somewhat more refined information about the coefficients that make up the identities for $S_{s,i}$ as well as identities that are valid for all $i \geqslant 0$.  Thus, instead of \cite[Thm.~7]{PellarinPerkins20} we build on identities and techniques from the proof of \cite[Thm.~2]{PellarinPerkins20}.

\begin{remark}
Demeslay~\cite[Thm.~3.3.6]{DemeslayPhD} has derived remarkable identities for $S_{s,i}$ for all $s \geqslant 1$, which could shed light on the developments here.  However, Demeslay's results are derived from multivariable log-algebraicity identities of Angl\`{e}s, Pellarin, and Tavares Ribeiro~\cite{AnglesPellarinTavares18}, which themselves are generalizations of Anderson's Theorem~\ref{T:Anderson}.  Because our purpose in this section is to give an independent proof of Theorem~\ref{T:Anderson}, we do not explore Demeslay's results here.  For more information about Demeslay's identities, the reader is also directed to Gezmi\c{s}~\cite[\S 2.1--2]{Gezmis19}.
\end{remark}

We review some of the essential notation from Pellarin and Perkins~\cite{PellarinPerkins20}, but for expedience we will assume that the reader has some facility with the notation and constructions there.  When possible we have tried to use their notation, unless we have already defined notation for the same object (e.g., their $b_i(t)$ is $\mu_i(t)$ of \eqref{E:mudef} in the present paper).  Recalling the polynomials $E_i(z) \in K[z]$ from~\eqref{E:EbracCarlitz}, for $j = j_0 + j_1 q + \cdots + j_u q^{u}$ with $0 \leqslant j_k \leqslant q-1$, we set
\[
  G_j(z) \assign E_0^{j_0} \cdots E_{u}^{j_{u}} \in K[z],
\]
with $G_0 = 1$.  Since $\deg_z(G_j) = j$, the polynomials $\{ G_j \}$ form a $K$-basis of $K[z]$.  For $\uell = (\ell_1, \dots, \ell_{\nu}) \in \ZZ_{\geqslant 0}^{\nu}$, we can uniquely define a sequence $\{ \epsilon_{\uell,j} \} \subseteq K$ so that
\begin{equation} \label{E:Eproddec}
  E_{\ell_1} \cdots E_{\ell_{\nu}} = \sum_{j \geqslant 0} \epsilon_{\uell,j} G_j.
\end{equation}
We note that $\epsilon_{\uell,j} = 0$ for $j > \deg_z(E_{\ell_1} \cdots E_{\ell_{\nu}}) = q^{\ell_1} + \dots + q^{\ell_{\nu}}$.  As necessary we allow $\nu$ to vary.

Now fix $\sigma \in \ZZ$ with $s < \sigma$ so that $\sigma = 1 + \rho(q-1)$ with $\rho \geqslant 2$, and let $i \geqslant 0$.  As in the proof of~\cite[Thm.~7]{PellarinPerkins20}, we observe that $S_{s,i}$ is the coefficient of $t_{s+1}^i \cdots t_{\sigma}^i$ in $F_{\sigma,i+1}$.  Moreover, from~\cite[\S 2.1.2]{PellarinPerkins20}, Pellarin and Perkins prove
\begin{equation} \label{E:Fsigmasum1}
  F_{\sigma,i+1} = -\sum_{h=0}^{\lfloor \log_q(\sigma) \rfloor - 1} \frac{1}{L_{i+1+h}}
  \sum_{(\ell_1, \ldots, \ell_{\sigma}) \leqslant i} \epsilon_{(\ell_1, \ldots, \ell_{\sigma}),q^{i+1+h}} \mu_{\ell_1}(t_1) \cdots \mu_{\ell_{\sigma}}(t_{\sigma}), \quad i \geqslant 0,
\end{equation}
where $\uell \leqslant i$ indicates that each entry of $\uell$ is at most~$i$.  If $i \geqslant \rho-1$, then as Pellarin and Perkins observe, $F_{\sigma,i+1}|_{t_i = \theta^{q^{\nu}}} = 0$ for $0 \leqslant \nu \leqslant i-\rho$, and so this identity simplifies as
\begin{align} \label{E:Fsigmasum2}
    F_{\sigma,i+1} = -&\sum_{h=0}^{\lfloor \log_q(\sigma) \rfloor - 1} \frac{1}{L_{i+1+h}} \\
 &{} \times \sum_{i+1 -\rho \leqslant (\ell_1, \ldots, \ell_{\sigma}) \leqslant i} \epsilon_{(\ell_1, \ldots, \ell_{\sigma}),q^{i+1+h}} \mu_{\ell_1}(t_1) \cdots \mu_{\ell_{\sigma}}(t_{\sigma}), \quad i \geqslant \rho-1. \notag
\end{align}
Our preliminary goal is to build on~\cite[Lem.~9]{PellarinPerkins20} to obtain more detailed information about the coefficients $\epsilon_{(\ell_1, \ldots, \ell_{\sigma}),q^{k}}$.

To this end we develop a kind of combinatorial game for working with decompositions of products of the form $E_{\ell_1} \cdots E_{\ell_{\nu}}$ in \eqref{E:Eproddec}.  Let
\[
  \cS \assign \bigl\{ (w_0, w_1, \ldots ) \in \ZZ_{\geqslant 0}^{\infty} \mid \forall\, k \gg 0,\ w_k = 0 \bigr\},
\]
and for a new variable $Y$, let $\cW$ be the free $A[Y]$-module on $\cS$.  For $\bw = (w_k) \in \cS$, we say $\bw$ is reduced if $w_k \leqslant q-1$ for all $k$.  Otherwise, we let $\kappa=\kappa(\bw)$ denote the smallest index such that $w_{\kappa} \geqslant q$.  We then define functions $u$, $v : \cS \to \cS$ and $m : \cS \to \cW$ by
\begin{align*}
  u(\bw) &\assign \begin{cases}
  \bw & \textup{if $\bw$ is reduced,} \\
  (w_0, \dots, w_{\kappa-1}, w_{\kappa}-(q-1), w_{\kappa+1}, \ldots ) & \textup{otherwise,}
  \end{cases} \\
  v(\bw) &\assign \begin{cases}
  \bw & \textup{if $\bw$ is reduced,} \\
  (w_0, \dots, w_{\kappa-1}, w_{\kappa}-q, w_{\kappa+1}+1, w_{\kappa+2}, \ldots) & \textup{otherwise,}
  \end{cases} \\
  m(\bw) &\assign \begin{cases}
  \bw & \textup{if $\bw$ is reduced,} \\
  u(\bw) + \bigl( Y^{q^{\kappa+1}} - \theta \bigr) v(\bw) & \textup{otherwise.}
  \end{cases}
\end{align*}
We extend $A[Y]$-linearly to obtain maps $u$, $v$, $m : \cW \to \cW$.  For $\bw = (w_k) \in \cS$ we further define the weight and size of $\bw$ to be the non-negative integers
\[
  \wt(\bw) = \sum_{k \geqslant 0} w_k q^k, \quad \sz(\bw) = \sum_{k \geqslant 0} w_k.
\]
It is readily apparent that if $\bw$ is not reduced, then
\begin{align} \label{E:szwt}
  \wt(u(\bw)) &= \wt(\bw) - (q-1)q^{\kappa} & \sz(u(\bw)) &= \sz(\bw) - (q-1), \\
  \wt(v(\bw)) &= \wt(\bw) & \sz(v(\bw)) &= \sz(\bw) - (q-1). \notag
\end{align}
We note that for $B > 0$,
\[
  \#\{ \bw \in \cS \mid \wt(\bw) < B\} < \infty.
\]
For arbitrary $\bx \in \cW$, we set $\wt(\bx)$ and $\sz(\bx)$ to be the maximum weight and size respectively of the generators from $\cS$ in the support of $\bx$, and sets of bounded weight in $\cW$ are supported on a finite set of generators in~$\cS$.

We think of $m$ as being the primary `move' in our game.  For any $\bx \in \cW$, if we let $m^j(\bx) \assign (m \circ \cdots \circ m)(\bx)$ be the $j$-fold iterate of $m$ applied to $\bx$, the sequence
\[
  m(\bx), m^2(\bx), m^3(\bx), \ldots
\]
is a sequence of bounded weight.  Moreover, it eventually stabilizes, since after finitely many applications of $m$ the size of $m^j(\bx)$ is necessarily strictly less than the size of $\bx$ but the weight of $m^j(\bx)$ is no more than the weight of $\bx$.  Thus eventually each generator from $\cS$ in the support of $m^j(\bx)$ is reduced, and we let $M(\bx)\in \cW$ denote this stabilization of this sequence.

\begin{example} \label{Ex:Mexamples}
The following calculations hold.  We let `$\ozero$' denote a sequence of infinitely many $0$'s.
\begin{align*}
M(q, \ozero) &= (1, \ozero) + \bigl( Y^{q} - \theta \bigr) (0,1,\ozero), \\
M(2q-1,\ozero) &= (1,\ozero) + \bigl(Y^{q} - \theta \bigr) (0,1,\ozero) + (Y^q-\theta)(q-1,1,\ozero), \\
M(0,q,\ozero) &= (0,1,\ozero) + \bigl(Y^{q^2}-\theta \bigr) (0,0,1,\ozero), \\
M(0,2q-1,\ozero) &=
  (0,1,\ozero) + \bigl(Y^{q^2} - \theta \bigr) (0,0,1,\ozero)
   + (Y^q-\theta)(0,q-1,1,\ozero), \\
M(q^2,\ozero) &= \begin{aligned}[t]
  (1,\,&\ozero) + \bigl( \bigl(Y^{q}-\theta \bigr) + \bigl(Y^{q}-\theta \bigr)^q \bigr) (0,1,\ozero)\\
  &{} + \bigl( Y^{q}-\theta \bigr)^q \bigl(Y^{q^2} - \theta \bigr) (0,0,1,\ozero),
  \end{aligned} \\
M(0,q^2,\ozero) &= \begin{aligned}[t]
  (0,\, &1,\ozero) + \bigl( \bigl(Y^{q^2}-\theta \bigr) + \bigl(Y^{q^2}-\theta \bigr)^q \bigr) (0,0,1,\ozero)\\
  &{} + \bigl( Y^{q^2}-\theta \bigr)^q \bigl(Y^{q^3} - \theta \bigr) (0,0,0,1,\ozero).
  \end{aligned}
\end{align*}
For $n \geqslant 0$,
\[
  M(q^n, \ozero) = (1, \ozero)
  + \sum_{k=1}^n \Biggl( \sum_{0 \leqslant j_k < \cdots < j_1 \leqslant n}
  \bigl( Y^{q}-\theta \bigr)^{q^{j_1}} \cdots \bigl(Y^{q^k}-\theta \bigr)^{q^{j_k}} \Biggr)
  (\underbrace{0, \ldots,0}_{k}, 1,\ozero).
\]
\end{example}

For $\bw = (w_k) \in \cS$, we set
\begin{equation} \label{E:ewdef}
  e(\bw) \assign \prod_{k \geqslant 0} E_k^{w_k},
\end{equation}
and we extend it $A[Y]$-linearly to a map $e : \cW \to K[z]$ by setting also $Y=\theta$.  We note that if $\bw \in \cS$ is reduced, then $e(\bw) = G_{\wt(\bw)}$.  Thus for $\bx = \sum_{\bw \in \cS} a_{\bw} \bw \in \cW$, with $a_{\bw} \in A[Y]$, if each $\bw$ in the support of $\bx$ is reduced, then
\begin{equation} \label{E:xexpansion}
  e(\bx) = \sum_{\bw \in \cS} a_{\bw}(\theta) G_{\wt(\bw)}.
\end{equation}
The motivation for this entire construction is to obtain identities for $\epsilon_{\uell,j}$  (see~\eqref{E:epsellj}) by combining \eqref{E:xexpansion} with the following lemma, which itself follows directly from~\cite[Prop.~8(3)]{PellarinPerkins20}.

\begin{lemma} \label{L:estabilize}
For any $\bx \in \cW$, we have
\[
  e(\bx) = e(m(\bx)) = e(M(\bx)).
\]
\end{lemma}

It will be useful to have some additional operators, inspired by the calculations in Example~\ref{Ex:Mexamples}.  We define shift operators $\alpha$, $\beta : \cS \to \cS$ so that
\[
\alpha(w_0, w_1, \ldots) \assign (0, w_0, w_1, \ldots), \quad
\beta(w_0, w_1, \ldots) \assign (w_1, w_2, \ldots),
\]
and then extend $A[Y]$-linearly to $\alpha$, $\beta : \cW \to \cW$.  Certainly, for all $n \geqslant 0$, we have $\beta^n \circ \alpha^n = \Id_{\cW}$, and if for $\bw \in \cS$, $w_0= \cdots = w_{n-1}=0$, then $\alpha^n \circ \beta^n(\bw) =\bw$.  We note that for $\bw \in \cS$ and $n \geqslant 0$,
\begin{equation} \label{E:Malphaw}
  M( \alpha^n(\bw)) = \alpha^n \bigl( M(\bw)\big|_{Y \leftarrow Y^{q^n}} \bigr) = \alpha^n \bigl( M(\bw) \bigr)\big|_{Y \leftarrow Y^{q^n}}.
\end{equation}
A word of caution is that this identity does not extend $A[Y]$-linearly to all of $\cW$, though it will not impact us here.  If $w_0 = \cdots = w_{n-1}=0$, then it also follows that
\begin{equation} \label{E:Mbetaw}
   M(\beta^n(\bw)) = \beta^n \bigl( M(\bw) \big|_{Y \leftarrow Y^{1/q^n}} \bigr) = \beta^n \bigl( M(\bw) \bigr) \big|_{Y \leftarrow Y^{1/q^n}}.
\end{equation}

Furthermore, we let $\cI \assign \cup_{\sigma \geqslant 1} \ZZ_{\geqslant 0}^{\sigma}$, taken as a disjoint union, and define $\xi : \cI \to \cS$ by setting
\[
  \xi(\ell_1, \dots, \ell_{\sigma}) \assign (\xi_0, \xi_1, \ldots),
\]
where $\xi_k \assign \# \{ i \mid \ell_i = k \}$.  (Previously we had assumed $\sigma \equiv 1 \pmod{q-1}$, but we do not need to continue with that assumption until later.)  Our main computational tool is then the following.  For $\uell \in \cI$ and for $j \geqslant 0$, expressed as $j = j_0 + j_1 q + \dots + j_u q^u$, $0 \leqslant j_i \leqslant q-1$, we set
\[
  m_{\uell,j}(Y) \assign \textup{coefficient of $(j_0, j_1, \ldots)$ in $M(\xi(\uell))$,}
\]
which is an element of $A[Y]$.  It then follows from \eqref{E:xexpansion} and Lemma~\ref{L:estabilize} that
\begin{equation} \label{E:epsellj}
  \epsilon_{\uell,j} = m_{\uell,j}(\theta).
\end{equation}
As it turns out, the polynomials $m_{\uell,j}(Y)$ are the same as another class of polynomials $c_{\uell,j}(Y) \in A[Y]$ defined in~\cite{PellarinPerkins20}.

\begin{lemma}[{Pellarin-Perkins~\cite[Lem.~9]{PellarinPerkins20}}] \label{L:Lemma9}
Let $\uell \in \cI$.  For each $j \geqslant 0$, there exists a polynomial $c_{\uell,j}(Y) \in A[Y]$, all but finitely many of which are non-zero, so that for all $n \geqslant 0$,
\[
  e(\alpha^n(\xi(\uell))) = \sum_{j \geqslant 0} c_{\uell,j}\bigl( \theta^{q^n} \bigr) G_{jq^n}.
\]
\end{lemma}

\begin{corollary} \label{C:mandc}
With notation as above, $m_{\uell,j}(Y) = c_{\uell,j}(Y)$.
\end{corollary}

\begin{proof}
It suffices to show that $m_{\uell,j}(Y)$ and $c_{\uell,j}(Y)$ agree at $Y = \theta^{q^n}$ for all $n \geqslant 0$.  By~\eqref{E:Eproddec} and~\eqref{E:epsellj}, we see that $m_{\uell,j}(\theta) = c_{\uell,j}(\theta)$.  Moreover, for $n \geqslant 1$, letting $\uell + n$ denote the tuple $(\ell_1 + n, \ldots, \ell_{\sigma}+n)$ and using~\eqref{E:Malphaw} and Lemma~\ref{L:Lemma9}, we have
\begin{align} \label{E:cmequal}
c_{\uell,j}\bigl(\theta^{q^n} \bigr) &= \epsilon_{\uell+n,jq^n} \\
&= \textup{coefficient of $\alpha^n(j_1, \ldots, j_u, \ozero)$ in $M(\alpha^n(\xi(\uell)))$ with $Y=\theta$,}
\notag \\
&= \textup{coefficient of $(j_0, \ldots, j_u, \ozero)$ in $M(\xi(\uell))$ with $Y= \theta^{q^n}$,} \notag \\
&= m_{\uell,j} \bigl(\theta^{q^n} \bigr). \notag
\end{align}
\end{proof}

Henceforth we will use $c_{\uell,j}(Y)$ exclusively to denote $m_{\uell,j}(Y)$.  To complete our calculations for \eqref{E:Fsigmasum1}--\eqref{E:Fsigmasum2}, we have the following lemma, whose proof is a variant of the proof of Lemma~\ref{L:Lemma9} in~\cite{PellarinPerkins20}.

\begin{lemma} \label{L:finaldivisibility}
Let $\uell \in \cI$, and suppose $\xi(\uell) = (\xi_0, \xi_1, \ldots)$ with $\xi_0 = \cdots = \xi_{k_0-1} = 0$ and $\xi_{k_0}\neq 0$ for some $k_0 \geqslant 0$.  Then for any $k > k_0$,
\[
\bigl(Y^{q^{k_0+1}}-\theta \bigr) \bigl(Y^{q^{k_0+2}}-\theta \bigr) \cdots \bigl(Y^{q^k}-\theta \bigr) \quad
\textup{divides} \quad c_{\uell,q^k}(Y)
\]
in $A[Y]$.
\end{lemma}

\begin{proof}
For $k \geqslant 0$, let $\bs_k \in \cS$ be given by $\bs_k \assign (0, \ldots, 0, 1, \ozero)$ where the~$1$ is in the $k$-th entry of $\bs_k$.  The key thing to keep to in mind is that by the definition of $m_{\uell,q^k}$, it follows that $c_{\uell,q^k}(Y)$ is the coefficient of $\bs_k$ in $M(\xi(\uell))$.  We proceed by induction on the size $\sz(\xi(\uell)) = \sum_i \xi_i$.  If $\sz(\xi(\uell)) = 1$, then by the definition of $k_0$ we have $\xi(\uell) = \bs_{k_0}$, and also $M(\xi(\uell)) = \xi(\uell)$.  Thus $c_{\uell,q^k}(Y) = 0$ identically for every $k > k_0$.

Now suppose $\sz(\xi(\uell)) \geqslant 2$ and that the statement holds for elements of size $< \sz(\xi(\uell))$.  If $1 \leqslant \xi_{k_0} \leqslant q-1$, then by the moves of our game, every element $\bw \in \cS$ in the support of $M(\xi(\uell))$ starts as
\[
  \bw = (\underbrace{0, \ldots, 0}_{k_0}, \xi_{k_0}, \ldots),
\]
and so we have $c_{\uell,q^k}(Y)=0$ identically for all $k > k_0$.

Therefore, suppose that $\xi_{k_0} \geqslant q$, and apply $m$ to $\xi(\uell)$:
\begin{align} \label{E:mxiell}
  m(\xi(\uell)) &= \begin{aligned}[t]
    \bigl(\underbrace{0,\ldots,0}_{k_0},\ &\xi_{k_0} - (q-1), \xi_{k_0+1}, \ldots \bigr) \\
  &{} + \bigl(Y^{q^{k_0+1}} - \theta \bigr) \bigl( \underbrace{0, \ldots, 0}_{k_0}, \xi_{k_0} - q, \xi_{k_0+1}+1, \ldots \bigr)
  \end{aligned} \\
  &\rightassign \by + \bigl(Y^{q^{k_0+1}} - \theta \bigr)\bz. \notag
\end{align}
As noted in~\eqref{E:szwt}, the sizes of each term on the right is $\sz(\xi(\uell)) - (q-1)$, so we can apply our induction hypothesis.  Since $\xi_{k_0}-(q-1) \geqslant 1$, the induction hypothesis applied to $\by$ implies that for $k > k_0$, the coefficient of $\bs_k$ in $M(\by)$ is divisible by $(Y^{q^{k_0+1}} - \theta) \cdots (Y^{q^k}-\theta)$.

We consider $M(\bz)$.  If $\xi_{k_0} > q$, then the argument in the previous paragraph holds as well for $M(\bz)$.
So finally we suppose that $\xi_{k_0}=q$.  If $k = k_0+1$, then there is nothing to show since the $Y^{q^{k_0+1}}-\theta$ in front in~\eqref{E:mxiell} is the divisor we want and it persists throughout all calculations of $M(\xi(\uell)) = M(m(\xi(\uell)))$.  On the other hand, if $k > k_0 + 1$, since $\xi_{k_0+1}+1 \neq 0$ we can apply induction with $k_0$ replaced by $k_0+1$, and so the coefficient of $\bs_k$ in $M(\bz)$ is divisible by $(Y^{q^{k_0+2}} - \theta) \cdots (Y^{q^k}-\theta)$.  Again adding on the $Y^{q^{k_0+1}}-\theta$ term from~\eqref{E:mxiell}, we are done.
\end{proof}

Using Lemma~\ref{L:finaldivisibility} together with \eqref{E:Fsigmasum1}--\eqref{E:Fsigmasum2}, we now derive identities for $S_{s,i}$ for $i \geqslant 0$.  As earlier in the section we let $\sigma = 1 + \rho(q-1)$ with $\rho \geqslant 2$ and assume $s < \sigma$.  Since $S_{s,i}$ is the coefficient of $t_{s+1}^i \cdots t_{\sigma}^i$ in $F_{\sigma,i+1}$, we see from~\eqref{E:Fsigmasum1} that
\[
  S_{s,i} = - \sum_{h=0}^{\lfloor \log_q(\sigma) \rfloor - 1} \frac{1}{L_{i+1+h}}
  \sum_{\uell \leqslant i} \epsilon_{(\uell,\oi),q^{i+1+h}}
  \mu_{\ell_1}(t_1) \cdots \mu_{\ell_s}(t_s),
\]
where $\uell \in \ZZ_{\geqslant 0}^s$ and $\oi$ represents a string of $i$'s repeated $\sigma-s$ times (so $(\uell,\oi) \in \ZZ_{\geqslant 0}^{\sigma}$).  As we saw in~\eqref{E:Fsigmasum2}, if $i - \rho + 1 \not\leqslant \uell$, then $\epsilon_{(\uell,\oi),q^{i+1+h}} = 0$.  Thus if we let $\delta(\uell)$ denote the minimum of the entries of $\uell$, then for all $i \geqslant 0$,
\begin{equation} \label{E:Ssiintermediate}
  S_{s,i} = - \sum_{h=0}^{\lfloor \log_q(\sigma) \rfloor - 1} \frac{1}{L_{i+1+h}}
  \sum_{\nu=\max(1,\rho-i)}^\rho\; \sum_{\substack{i -\rho + \nu \leqslant \uell \leqslant i \\ \delta(\uell) = i-\rho+\nu}} \epsilon_{(\uell,\oi),q^{i+1+h}} \mu_{\ell_1}(t_1) \cdots \mu_{\ell_s}(t_s).
\end{equation}
If $\delta(\uell) = i - \rho + \nu \geqslant 0$, then Lemma~\ref{L:Lemma9} with $n = i - \rho + \nu$ yields (as in~\cite[\S 2.1.2]{PellarinPerkins20}, and see also~\eqref{E:cmequal} above)
\begin{equation} \label{E:epstoc}
  \epsilon_{(\uell,\oi),q^{i+1+h}} = c_{(\ell_1 - i + \rho - \nu, \ldots, \ell_s - i + \rho - \nu,\rho - \nu, \ldots, \rho - \nu),q^{h+\rho-\nu+1}}(Y)|_{Y = \theta^{q^{i-\rho + \nu}}}.
\end{equation}
This prompts the following definition.  For $1 \leqslant \nu \leqslant \rho$, suppose $\ueta \in \ZZ_{\geqslant 0}^{s}$ satisfies $0 \leqslant \ueta \leqslant \rho - \nu$ with $\delta(\ueta) = 0$.  Then for $0 \leqslant h \leqslant \lfloor \log_q(s) \rfloor -1$, Lemma~\ref{L:finaldivisibility} enables us to define a polynomial in~$A[Y]$,
\begin{equation} \label{E:Deltadef}
  \Delta_{\ueta,\nu,h}(Y) \assign \frac{c_{(\ueta,\rho - \nu, \ldots \rho - \nu),q^{h+\rho-\nu+1}}(Y)}{(Y^q - \theta) \cdots (Y^{q^{h+\rho-\nu+1}} - \theta)} \in A[Y].
\end{equation}
Moreover, for $\uell = (\ell_1, \dots, \ell_s) \leqslant i$, for $1 \leqslant \nu \leqslant \rho$ with $\delta(\uell) = i - \rho + \nu \geqslant 0$, and for $0 \leqslant h \leqslant \lfloor \log_q(\sigma) \rfloor - 1$, we let $\eta_i(\uell) \assign \uell - (i-\rho+\nu) = (\ell_1-i+\rho-\nu, \dots, \ell_s - i + \rho -\nu) \in \{0,\dots, \rho-\nu\}^{s}$.  We then find from~\eqref{E:epstoc} that
\[
  -\frac{\epsilon_{(\uell,\oi),q^{i+1+h}}}{L_{i+1+h}} = (-1)^{h+\rho-\nu} \frac{ \Delta_{\eta_i(\uell),\nu,h}\bigl( \theta^{q^{i-\rho+\nu}} \bigr)}{L_{i - \rho + \nu}}.
\]
Combining this identity with~\eqref{E:Ssiintermediate} we prove the following lemma, after which we can finally prove Theorem~\ref{T:Anderson}.

\begin{lemma} \label{L:Ssi}
Let $\sigma = 1+\rho(q-1)$ with $\rho \geqslant 2$, and let $s < \sigma$.  Then for $i \geqslant 0$,
\begin{align*}
  S_{s,i} &= \sum_{h=0}^{\lfloor \log_q(\sigma) \rfloor -1} \sum_{\nu=\max(1,\rho - i)}^{\rho} (-1)^{h+\rho-\nu}
  \sum_{\substack{i - \rho + \nu \leqslant \uell \leqslant i \\ \delta(\uell) = i - \rho + \nu }}
  \frac{\Delta_{\eta_i(\uell),\nu,h}\bigl( \theta^{q^{i-\rho+\nu}} \bigr)}{L_{i-\rho+\nu}} \cdot
  \mu_{\ell_1}(t_1) \cdots \mu_{\ell_s}(t_s), \\
  &= \begin{aligned}[t]
    \sum_{h=0}^{\lfloor \log_q(\sigma) \rfloor -1} \sum_{\nu=\max(1,\rho - i)}^{\rho} (-1)^{h+\rho-\nu}
    \sum_{\substack{0 \leqslant \ueta \leqslant \rho -\nu \\ \delta(\ueta) = 0}}
    &\frac{\Delta_{\ueta,\nu,h}\bigl( \theta^{q^{i-\rho+\nu}} \bigr)}{L_{i-\rho+\nu}} \\
    &{} \times \mu_{\eta_1+i-\rho+\nu}(t_1) \cdots \mu_{\eta_s + i - \rho + \nu}(t_s).
  \end{aligned}
\end{align*}
\end{lemma}

\begin{proof}[Proof of Theorem~\ref{T:Anderson}]
As in the beginning of the present section, we take $m=s$, and for $i \geqslant 0$, we consider $\lambda_i(s)$ with $\mu_1 = \cdots = \mu_s = 0$ in~\eqref{E:lambdaim} so that $P_s(x,z) = \exp_C(\sum_{i=0}^{\infty} \lambda_i(s) z^{q^i})$.  The formula in~\eqref{E:bracsum} remains valid in this more general case, and we see from~\eqref{E:hypersumred1} and the argument in~\eqref{E:Ssired1} that
\begin{align*}
  \sum_{a \in A_{i+}} \frac{\pd_{\theta}^{j_1}(a) \cdots \pd_{\theta}^{j_s}(a)}{a} &=
  \bigl( \pd_{t_1}^{j_1} \circ \cdots \circ \pd_{t_s}^{j_s}\bigr) \bigl( S_{s,i} \bigr) \big|_{t_1=\theta, \ldots, t_s=\theta} \\
  &=\begin{aligned}[t]
  \sum_{h=0}^{\lfloor \log_q(\sigma) \rfloor -1} &\sum_{\nu=\max(1,\rho - i)}^{\rho} (-1)^{h+\rho-\nu}
  \sum_{\substack{i - \rho + \nu \leqslant \uell \leqslant i \\ \delta(\uell) = i - \rho + \nu }}
  \frac{\Delta_{\eta_i(\uell),\nu,h}\bigl( \theta^{q^{i-\rho+\nu}} \bigr)}{L_{i-\rho+\nu}} \\
  &{}\times \prod_{r=1}^s (-1)^{\ell_r-j_r} e_{\ell_r-1,\ell_r-j_r}([1], \dots, [\ell_r-1]).
  \end{aligned}
\end{align*}
We note that $e_{\ell_r-1,\ell_r-j_r}=0$ unless $j_r \leqslant \ell_r \leqslant i$, and so when we substitute this expression into~\eqref{E:bracsum} and reorder the sum, we obtain an interior sum of the form
\begin{multline*}
  \prod_{r=1}^{s} \sum_{j_r=k_r}^{\ell_r} (-1)^{\ell_r-j_r} e_{\ell_r-1,\ell_r-j_r}([1],\dots, [\ell_r-1]) \cdot h_{k_r,j_r-k_r}([1], \dots, [k_r])\\
  = \begin{cases}
  1 & \textup{if $k_1=\ell_1, \ldots, k_s=\ell_s$,} \\
  0 & \textup{otherwise,}
  \end{cases}
\end{multline*}
where the equality follows from Proposition~\ref{P:symmrec2}.  Combining all of this with~\eqref{E:lambdaim} and the definition of $\eta_i(\uell)$, we obtain
\begin{align*}
  \lambda_i(s) &= \sum_{h=0}^{\lfloor \log_q(\sigma) \rfloor -1} \sum_{\nu=\max(1,\rho - i)}^{\rho} (-1)^{h+\rho-\nu}
  \sum_{\substack{i - \rho + \nu \leqslant \uell \leqslant i \\ \delta(\uell) = i - \rho + \nu }}
  \frac{\Delta_{\eta_i(\uell),\nu,h}\bigl( \theta^{q^{i-\rho+\nu}} \bigr)}{L_{i-\rho+\nu}} \cdot x^{q^{\ell_1} + \cdots + q^{\ell_s}} \\
  &= \sum_{h=0}^{\lfloor \log_q(\sigma) \rfloor -1} \sum_{\nu=\max(1,\rho - i)}^{\rho} (-1)^{h+\rho-\nu}
  \sum_{\substack{0 \leqslant \ueta \leqslant \rho -\nu \\ \delta(\ueta) = 0 }}
  \frac{\Delta_{\ueta,\nu,h} \bigl(\theta^{q^{i-\rho+\nu}} \bigr)}{L_{i -\rho + \nu}} \cdot \bigl( x^{q^{\eta_1} + \dots + q^{\eta_s}} \bigr)^{q^{i-\rho + \nu}}.
\end{align*}
By reordering the sum and making the substitution $i \leftarrow i + \rho -\nu$, we find
\[
\sum_{i = 0}^{\infty} \lambda_i(s) z^{q^i} =
\sum_{h=0}^{\lfloor \log_q(\sigma) \rfloor -1} \sum_{\nu=1}^{\rho} (-1)^{h+\rho-\nu}
  \sum_{\substack{0 \leqslant \ueta \leqslant \rho -\nu \\ \delta(\ueta) = 0 }}\; \sum_{i \geqslant 0}
  \frac{\Delta_{\ueta,\nu,h} \bigl(\theta^{q^{i}} \bigr)}{L_{i}} \cdot \bigl( x^{q^{\eta_1} + \dots + q^{\eta_s}} \bigr)^{q^{i}} \bigl(z^{q^{\rho-\nu}}\bigr)^{q^i}.
\]
If $\Delta_{\ueta,\nu,h}(Y) = \sum_{k=0}^d a_{\ueta,\nu,h,k}Y^k$ with $a_{\ueta,\nu,h,k} \in A$, where $d$ is taken to be the maximum of the degrees in $Y$ of the polynomials $\Delta_{\ueta,\nu,h}(Y)$ that appear in the sum above, then this becomes
\begin{align} \label{E:finalstep}
\sum_{i=0}^{\infty} \lambda_i(s) z^{q^i}
  =\sum_{h=0}^{\lfloor \log_q(\sigma) \rfloor -1} &\sum_{\nu=1}^{\rho} (-1)^{h+\rho-\nu}
  \sum_{\substack{0 \leqslant \ueta \leqslant \rho -\nu \\ \delta(\ueta) = 0 }} \\
  &{}\times \sum_{k=0}^d
  a_{\ueta,\nu,h,k} \log_C \Bigl( \theta^k \cdot x^{q^{\eta_1} + \dots + q^{\eta_s}} \cdot z^{q^{\rho-\nu}}\Bigr), \notag
\end{align}
and the result follows upon exponentiation.
\end{proof}

\begin{example}
We demonstrate elements of the proof of Theorem~\ref{T:Anderson} in the case $s = q+1$ for $q>2$.  We take $\sigma = 2q-1$, and so $\rho = 2$.  We note that $h=0$ throughout this calculation.  For $\nu = 2$ and for $\ueta = (0,\dots, 0) \in \ZZ_{\geqslant 0}^{q+1}$, we see that $\xi(\ueta,0, \ldots, 0) = (2q-1,\ozero) \in \cS$, and the calculations in Example~\ref{Ex:Mexamples} together with \eqref{E:Deltadef} show that
\[
  \Delta_{(0,\ldots,0),2,0}(Y) = 1.
\]
Likewise for $\nu=1$ and for $\ueta = (0,\dots,0,1) \in \ZZ_{\geqslant 0}^{q+1}$, we see that $\xi(\ueta,1, \ldots, 1) = (q,q-1,\ozero)\in \cS$.  Similar to the calculations in Example~\ref{Ex:Mexamples}, we find
\[
  M(q,q-1,\ozero) = (1,q-1,\ozero) + (Y^q-\theta)(0,1,\ozero) + (Y^q-\theta)(Y^{q^2}-\theta)(0,0,1,\ozero),
\]
from which we find using the coefficient of $(0,0,1,\ozero)$ and \eqref{E:Deltadef} that also
\[
  \Delta_{(0,\ldots,0,1),1,0}(Y) = 1.
\]
We can permute $(0,\ldots,0,1)$ in $q+1$ ways and obtain the same polynomial, but other than that there are no other non-zero $\Delta_{\ueta,\nu,0}(Y)$ in this case.  Then~\eqref{E:finalstep} becomes
\[
  \sum_{i=0}^{\infty} \lambda_i(q+1) z^{q^i} = \log_C \bigl( x^{q+1} z \bigr) - (q+1)\log_C\bigl( x^{q\cdot 1 + q} z^q \bigr),
\]
and so $P_{q+1}(x,z) = x^{q+1}z - x^{2q}z^q$, which agrees with Example~\ref{Ex:Thakurexamples}.
\end{example}

\begin{example}
We also consider the case $s=2q$ for $q > 2$.  We take $\sigma = 3q-2$, and so $\rho=3$.  As in the previous example $h=0$ throughout.  There are four types of choices of $\ueta \in \ZZ_{\geqslant 0}^{2q}$ that produce non-zero polynomials $\Delta_{\ueta,\nu,0}(Y)$.  For $\nu=3$, we take $\ueta_0 = (0,\ldots, 0)$, for which $\xi(\ueta,0,\ldots,0) = (3q-2,\ozero) \in \cS$.  We calculate
\[
  M(3q-2,\ozero) = (1, \ozero) + (Y^q-\theta)(0,1,\ozero) + 2(Y^q-\theta)(q-1,1,\ozero) + (Y^q-\theta)^2(q-2,2,\ozero),
\]
and using the coefficient of $(0,1,\ozero)$ in \eqref{E:Deltadef} we see that $\Delta_{\smash{\ueta}_0,3,0}(Y) = 1$.

For $\nu=2$, we let $\smash{\ueta}_{a,b} = (0,\dots, 0,1,\dots,1)$, with $0$ given $a$ times and $1$ given $b$ times and $a+b=2q$.  For $\xi(\smash{\ueta}_{q,q}) = (q,2q-2,\ozero)$, we find
\begin{multline*}
  M(q,2q-2,\ozero) = (1,q-1,\ozero) + (Y^{q^2}-\theta)(1,q-2,1,\ozero) + (Y^q-\theta)(0,1,\ozero) \\
   {}+ (Y^q-\theta)(Y^{q^2}-\theta)(0,0,1,\ozero) + (Y^q-\theta)^2(0,q-1,1,\ozero),
\end{multline*}
and using the coefficient of $(0,0,1,\ozero)$ in \eqref{E:Deltadef} we have $\Delta_{\smash{\ueta}_{q,q},2,0}(Y) = 1$.  Similarly, we have $\Delta_{\smash{\ueta}_{2q-1,1},2,0}(Y) = 1$ and $\Delta_{\smash{\ueta}_{2q,0},2,0}(Y) = Y^q-\theta$, but we omit the details.

Each of these four types of $\Delta_{\ueta,\nu,0}(Y)$ polynomials are unchanged by permuting the entries of $\ueta$, but beyond that all other $\Delta_{\ueta,\nu,0}(Y)=0$, including for all $\ueta$ with $\nu=1$.  Taking into account the possible permutations, $\Delta_{\smash{\ueta}_0,3,0}(Y)$ occurs once; $\Delta_{\smash{\ueta}_{q,q},2,0}(Y)$ occurs $\binom{2q}{q} \equiv 2 \pmod{p}$ times; $\Delta_{\smash{\ueta}_{2q-1,1},2,0}(Y)$ occurs $\binom{2q}{1} \equiv 0 \pmod{p}$ times; and finally, $\Delta_{\smash{\ueta}_{2q,0},2,0}(Y)$ occurs once.  Assembling all of this information into~\eqref{E:finalstep}, we have
\[
  \sum_{i=0}^{\infty} \lambda_i(2q) z^{q^i} = \log_C \bigl(x^{2q}z\bigr) - 2 \log_C \bigl(x^{q^2+q}z^q \bigr)
  + 0 + \theta \log_C\bigl( x^{2q}z^q \bigr) - \log_C \bigl( \theta^q x^{2q}z^q \bigr).
\]
After exponentiating and simplifying, we find
\begin{align*}
  P_{2q}(x,z) &= x^{2q}z - 2 x^{q^2+q}z^q + C_{\theta} \bigl(x^{2q}z^q \bigr) - \theta^{q} x^{2q} z^q \\
  &= x^{2q}z - \bigl(2x^{q^2+q} - \theta x^{2q} + \theta^q x^{2q} \bigr)z^q + x^{2q^2} z^{q^2},
\end{align*}
which agrees with Example~\ref{Ex:Thakurexamples}.
\end{example}


\begin{thebibliography}{99}

\bibitem{And86} %
G. W. Anderson, \textit{$t$-motives}, Duke Math. J. \textbf{53} (1986), no. 2, 457--502.

\bibitem{And96} %
G. W. Anderson, \textit{Log-algebraicity of twisted $A$-harmonic series and special values of $L$-series in characteristic $p$}, J. Number Theory \textbf{60} (1996), no. 1, 165--209.

\bibitem{AnglesPellarin14} %
B. Angl\`{e}s and F. Pellarin, \textit{Functional identities for $L$-series in positive characteristic}, J. Number Theory \textbf{142} (2014), 223--251.

\bibitem{AnglesPellarinTavares18} %
B. Angl\`{e}s, F. Pellarin, and F. Tavares Ribeiro, \textit{Anderson-Stark units for $\mathbb{F}_q[\theta]$}, Trans. Amer. Math. Soc. \textbf{370} (2018), no.~3, 1603--1627.

\bibitem{AnglesTavares17} %
B. Angl\`{e}s and F. Tavares Ribeiro, \textit{Arithmetic of function field units}, Math. Ann. \textbf{367} (2017), no.~1--2, 501--579.

\bibitem{BosserPellarin08} %
V. Bosser and F. Pellarin, \textit{Hyperdifferential properties of Drinfeld quasi-modular forms}, Int. Math. Res. Not. IMRN \textbf{2008} (2008), Art. ID rnn032, 56 pp.

\bibitem{BrownawellDenis00} %
W. D. Brownawell and L. Denis, \textit{Linear independence and divided derivatives of a Drinfeld module. II.} Proc. Amer. Math. Soc. \textbf{128} (2000), no.~6, 1581--1593.

\bibitem{BPrapid} %
W. D. Brownawell and M. A. Papanikolas, \textit{A rapid introduction to Drinfeld modules, $t$-modules, and $t$-motives}, in: $t$-Motives: Hodge Structures, Transcendence and Other Motivic Aspects, Eur. Math. Soc., Z\"urich, 2020, pp.~3--30.

\bibitem{Carlitz35} %
L. Carlitz, \textit{On certain functions connected with polynomials in a Galois field}, Duke Math. J. \textbf{1} (1935), no.~2, 137--168.

\bibitem{Carlitz37} %
L. Carlitz, \textit{An analogue of the von Staudt-Clausen theorem}, Duke Math. J. \textbf{3} (1937), no.~3, 503--517.

\bibitem{Carlitz39} %
L. Carlitz, \textit{Some sums involving polynomials in a Galois field}, Duke Math. J. \textbf{5} (1939), no.~4, 941--947.

\bibitem{CEP18} %
C.-Y. Chang, A. El-Guindy, and M. A. Papanikolas, \textit{Log-algebraic identities on Drinfeld modules and special $L$-values}, J. Lond. Math. Soc. (2) \textbf{97} (2018), no.~2, 125--144.

\bibitem{Conrad00} %
K. Conrad, \textit{The digit principle}, J. Number Theory \textbf{84} (2000), no.~2, 230--257.

\bibitem{DemeslayPhD} %
F. Demeslay, \textit{Formules de classes en caract\'{e}ristique positive}, Th\`{e}se de doctorat, Universit\'{e} de Caen Basse-Normandie, 2015.

\bibitem{Gekeler88a} %
E.-U. Gekeler, \textit{On power sums of polynomials over finite fields}, J. Number Theory \textbf{30} (1988), no.~1, 11--26.

\bibitem{Gezmis19} %
O. Gezmi\c{s}, \textit{Deformation of multiple zeta values and their logarithmic interpretation in positive characteristic}, arXiv:1910.02805, 2019.

\bibitem{Goss78} %
D. Goss, \textit{von Staudt for $\mathbf{F}_q[T]$}, Duke Math. J. \textbf{45} (1978), no.~4, 885--910.

\bibitem{Goss} %
D. Goss, \textit{Basic Structures of Function Field Arithmetic}, Springer-Verlag, Berlin, 1996.

\bibitem{Goss13blog} %
D. Goss, \textit{Wronski, Vandermonde, and Moore!}, Noncommutative Geometry, April 19, 2013, http://\allowbreak noncommutativegeometry.blogspot.com/2013/04/wronski-vandermonde-\allowbreak and-\allowbreak moore\_19\allowbreak .html.

\bibitem{Goss17} %
D. Goss, \textit{Digit permutations revisited}, J. Th\'{e}or. Nombres Bordeaux \textbf{29} (2017), no.~3, 693--728.

\bibitem{GreenP18} %
N. Green and M. A. Papanikolas, \textit{Special $L$-values and shtuka functions for Drinfeld modules on elliptic curves}, Res. Math. Sci. \textbf{5} (2018), 5:4, 47 pp.

\bibitem{Jeong00} %
S. Jeong, \textit{A comparison of the Carlitz and digit derivative bases in function field arithmetic}, J. Number Theory \textbf{84} (2000), no.~2, 258--275.

\bibitem{Jeong11} %
S. Jeong, \textit{Calculus in positive characteristic $p$}, J. Number Theory \textbf{131} (2011), no.~6, 1089--1104.

\bibitem{Knuth} %
D. E. Knuth, \textit{The Art of Computer Programming, Vol. 1, 2nd Ed.}, Addison-Wesley, Reading, MA, 1973.

\bibitem{Lee43} %
H. L. Lee, \textit{Power sums of polynomials in a Galois field}, Duke Math. J. \textbf{10} (1943), no.~2, 277--292.

\bibitem{Maurischat18} %
A. Maurischat, \textit{Prolongations of $t$-motives and algebraic independence of periods}, Doc. Math. \textbf{23} (2018), 815--838.

\bibitem{Mumford78} %
D. Mumford, \textit{An algebro-geometric construction of commuting operators and of solutions to the {T}oda lattice equation, {K}orteweg de{V}ries equation and related nonlinear equations}, in: Proceedings of the International Symposium on Algebraic Geometry (Kyoto Univ., Kyoto, 1977), Kinokuniya Book Store, Tokyo, 1978, pp.~115--153.

\bibitem{Okugawa} %
K. Okugawa, \textit{Differential algebra of nonzero characteristic}, Kinokuniya Co. Ltd., Tokyo, 1987.

\bibitem{OrucAkmaz04} %
H. Oru\c{c} and H. K. Akmaz, \textit{Symmetric functions and the Vandermonde matrix}, J. Comput. Appl. Math. \textbf{172} (2004), no.~1, 49--64.

\bibitem{OrucPhillips00} %
H. Oru\c{c} and G. M. Phillips, \textit{Explicit factorization of the Vandermonde matrix}, Linear Algebra Appl. \textbf{315} (2000), no.~1-3, 113--123.

\bibitem{PLogAlg} %
M. A. Papanikolas, \textit{Log-algebraicity on tensor powers of the Carlitz module and special values of Goss $L$-functions}, in preparation.

\bibitem{PZeng17} %
M. A. Papanikolas and G. Zeng, \textit{Theta operators, Goss polynomials, and $v$-adic modular forms}, J. Th\'{e}or. Nombres Bordeaux \textbf{29} (2017), no.~3, 729--753.

\bibitem{Pellarin12} %
F. Pellarin, \textit{Values of certain $L$-series in positive characteristic}, Ann. of Math. (2) \textbf{176} (2012), no.~3, 2055--2093.

\bibitem{PellarinPerkins20} %
F. Pellarin and R. Perkins, \textit{On twisted $A$-harmonic sums and Carlitz finite zeta values}, J. Number Theory (in press), 2019, https://doi.org/10.1016/j.jnt.2018.10.018.

\bibitem{Perkins14a} %
R. B. Perkins, \textit{An exact degree for multivariate special polynomials}, J. Number Theory \textbf{142} (2014), 252--263.

\bibitem{Perkins14c} %
R. B. Perkins, \textit{On Pellarin's $L$-series}, Proc. Amer. Math. Soc. \textbf{142} (2014), no.~10, 3355--3368.

\bibitem{Stanley} %
R. P. Stanley, \textit{Enumerative Combinatorics, Vol. 2}, Cambridge Univ. Press, Cambridge, 1999.

\bibitem{Thakur92}
D. S. Thakur, \textit{Drinfeld modules and arithmetic in the function fields}, Internat. Math. Res. Notices \textbf{1992} (1992), no. 9, 185--197.

\bibitem{Thakur93} %
D. S. Thakur, \textit{Shtukas and Jacobi sums}, Invent. Math. \textbf{111} (1993), no. 3, 557--570.

\bibitem{Thakur} %
D. S. Thakur, \textit{Function Field Arithmetic}, World Scientific Publishing, River Edge, NJ, 2004.

\bibitem{Thakur09} %
D. S. Thakur, \textit{Power sums with applications to multizeta and zeta zero distribution for $\FF_q[t]$}, Finite Fields Appl. \textbf{15} (2009), no.~4, 534--552.

\bibitem{Voloch98} %
J. F. Voloch, \textit{Differential operators and interpolation series in power series fields}, J. Number Theory \textbf{71} (1998), no.~1, 106--108.


\end{thebibliography}
\end{document}